\documentclass[11pt,letterpapert]{article}
\usepackage{graphicx}
\usepackage{amssymb}
\usepackage{epstopdf,mathtools}
\usepackage{fancyhdr}
\usepackage{amsmath}
\usepackage{amsthm}
\usepackage{hyperref}
\usepackage{makeidx}
\usepackage{blkarray}
\usepackage{mathdesign}
\usepackage{color}
\usepackage{geometry}
\usepackage{soul}
\usepackage{nameref}
\mathtoolsset{showonlyrefs}

\definecolor{purple}{rgb}{.9,0,.9}

\graphicspath{{./figures/}}

\def\smath#1{\text{\scalebox{.8}{$#1$}}}
\def\sfrac#1#2{\smath{\frac{#1}{#2}}}

\makeatletter
\let\orgdescriptionlabel\descriptionlabel
\renewcommand*{\descriptionlabel}[1]{%
  \let\orglabel\label
  \let\label\@gobble
  \phantomsection
  \edef\@currentlabel{#1}%
  \let\label\orglabel
  \orgdescriptionlabel{#1}%
}
\makeatother

\graphicspath{{./figures/}}

\newcommand{\Nat}{\mathbb{N}}

\newcommand{\Real}{\mathbb{R}}

\newcommand{\cC}{{\cal C}}
\newcommand{\cD}{{\cal D}}\newcommand{\cE}{{\cal E}}
\newcommand{\cH}{{\cal H}}

\newcommand{\cM}{{\cal M}}\newcommand{\cN}{{\cal N}}\newcommand{\cO}{{\cal O}}

\newcommand{\cS}{{\cal S}}\newcommand{\cU}{{\cal U}}
\newcommand{\cV}{{\cal V}}\newcommand{\cW}{{\cal W}}\newcommand{\cX}{{\cal X}}
\newcommand{\cY}{{\cal Y}}\newcommand{\cZ}{{\cal Z}}

\newcommand{\ba}{{\bf a}}\newcommand{\bb}{{\bf b}}
\newcommand{\be}{{\bf e}}

\newcommand{\bt}{{\bf t}}\newcommand{\bu}{{\bf u}}
\newcommand{\bv}{{\bf v}}\newcommand{\bx}{{\bf x}}
\newcommand{\by}{{\bf y}}\newcommand{\bA}{{\bf A}}
\newcommand{\bB}{{\bf B}}\newcommand{\bC}{{\bf C}}
\newcommand{\bE}{{\bf E}}\newcommand{\bF}{{\bf F}}
\newcommand{\bI}{{\bf I}}
\newcommand{\bL}{{\bf L}}\newcommand{\bM}{{\bf M}}
\newcommand{\bN}{{\bf N}}\newcommand{\bP}{{\bf P}}

\newcommand{\bU}{{\bf U}}\newcommand{\bV}{{\bf V}}
\newcommand{\bX}{{\bf X}}\newcommand{\bY}{{\bf Y}}

\newcommand{\Lin}{\mathop{\rm Lin}}

\newcommand{\SO}{\mathop{\rm SO}}

\newcommand{\blambda}{\boldsymbol{\lambda}}
\newcommand{\bnu}{\boldsymbol{\nu}}
\newcommand{\bmu}{\boldsymbol{\mu}}
\newcommand{\bomega}{\boldsymbol{\omega}}
\newcommand{\bchi}{\boldsymbol{\chi}}
\newcommand{\bPhi}{\boldsymbol{\Phi}}
\newcommand{\bphi}{\boldsymbol{\phi}}
\newcommand{\bpsi}{\boldsymbol{\psi}}

\newtheorem{theorem}{Theorem}[section]
\newtheorem{lemma}[theorem]{Lemma}

\newtheorem{remark}[theorem]{Remark}

\newcommand{\beqn}{\begin{equation}}
\newcommand{\eeqn}{\end{equation}}

\newcommand{\bone}{\textbf{1}}
\newcommand{\bzero}{{\bf 0}}

\def\H{\mathcal H}
\def\M{\mathcal M}
\def\A{\mathcal A}
\def\W{\mathcal W}

\def\R{\mathbb R}

\def\e{\varepsilon}
\def\s{\sigma}

\def\Lin{{\rm Lin}}
\def\rng{{\rm Rng}\,}

\def\bom{\boldsymbol{\omega}}

\def\aa{\textbf{a}}
\def\uu{\textbf{u}}
\def\ee{\textbf{e}}
\def\tt{\textbf{t}}

\def\Om{\Omega}
\def\d{\text{ d}}

\def\pa{\partial}

\def\E{\mathcal{E}}

\newcommand{\vol}{\textbf{\rm \textbf{vol}}}

\newcommand{\D}{\mathcal{D}}

\newcommand{\diam}{\mathrm{diam}}

\def\tt{\mathbf{t}}
\def\aa{\mathbf{a}}

\title{A definition of fractional $k$-dimensional measure: bridging the gap between fractional length and fractional area}

\author{Cornelia Mihaila and Brian Seguin}

\begin{document}

\maketitle

\tableofcontents

\begin{abstract}
\noindent Here we introduce a fractional notion of $k$-dimensional measure, $0\leq k<n$, that depends on a parameter $\sigma$ that lies between $0$ and $1$.  When $k=n-1$ this coincides with the fractional notions of area and perimeter, and when $k=1$ this coincides with the fractional notion of length.  It is shown that, when multiplied by the factor $1-\sigma$, this $\sigma$-measure converges to the $k$-dimensional Hausdorff measure up to a multiplicative constant that is computed exactly.  We also mention several future directions of research that could be pursued using the fractional measure introduced.\\

\noindent \emph{Dedicated to the memory of David Seguin, whose brother will always miss his laugh that filled a room.}
\end{abstract}

\section{Introduction} 

Given a parameter $\sigma$ satisfying $0<\sigma<1$, consider the functional
\beqn
J_\sigma(u)\coloneqq \int_{\R^n}\int_{\R^n} \frac{|u(x)-u(y)|^2}{|x-y|^{n+\sigma}} dxdy
\eeqn
defined on functions of the form $u:\Real^n\rightarrow\Real$, which is the square of the fractional Sobolev seminorm $|u|_{H^{\sigma/2}}$.  In the case where $u$ is the characteristic function $\chi_E$ of an open set $E$,
\beqn
J_\sigma(\chi_E)=2\int_{\R^n}\int_{\R^n}\frac{\chi_E(x)\chi_{E^c}(y)}{|x-y|^{n+\sigma}} dxdy,
\eeqn
where $E^c$ is the complement of $E$.  This motivates the consideration of the fractional perimeter defined by
\beqn
\text{Per}_\sigma(E)\coloneqq\frac{1}{\alpha_{n-1}}\int_E\int_{E^c} |x-y|^{-n-\sigma} dxdy,
\eeqn
where $\alpha_{n-1}$ is the volume of a unit ball in $\R^{n-1}$.  This functional, without the prefactor, was first studied by Visintin \cite{V91} who was motivated by the desire to study generalized surface tension.  When $E$ is unbounded, this functional is generally infinite.  To deal with this case, one can introduce a bounded, open set $\Omega$ and speak about the fractional perimeter of $E$ relative to $\Omega$, as was done by Caffarelli, Roquejoffre, and Savin \cite{CRS10}.  This is given by
\begin{align}
\text{Per}_\sigma(E,\Omega)\coloneqq&\frac{1}{\alpha_{n-1}}\Big( \int_{E\cap\Omega}\int_{E^c}+\int_{E\cap\Omega^c}\int_{E^c\cap\Omega}\Big) |x-y|^{-n-\sigma} dxdy\\
\label{sper}=&\frac{1}{\alpha_{n-1}}\int_{E}\int_{E^c} \frac{\max\{\chi_\Omega(x),\chi_\Omega(y)\}}{|x-y|^{n+\sigma}} dxdy.
\end{align}
A minimizer of the fractional perimeter, called a $\sigma$-minimal surface, is defined as a set $E$ such that for any other set $F\subseteq\R^n$ satisfying $E\setminus\Omega=F\setminus\Omega$ we have
\beqn
\text{Per}_\sigma(E,\Omega)\leq \text{Per}_\sigma(F,\Omega).
\eeqn
Caffarelli and collaborators \cite{C09,CS10,CV11} motivated the fractional perimeter by considering problems in nonlocal diffusion and phase transitions.  There are other applications as well.  For example, Lombardini \cite{L19} found a link between when the fractional perimeter is finite for a given parameter $\sigma$ and the fractal dimension of $\partial E$.

There are numerous works studying the minimizers of $\text{Per}_\sigma$.  See, for example, \cite{CL13,DV18,FV17,SV13}.  Among other things, it is known that $\sigma$-minimal surfaces are smooth off of a set of dimension at most $n-8$ for $\sigma$ close to $1$.  This is in agreement with a well-known result for classical minimal surfaces \cite{G84}.  However, $\sigma$-minimal surfaces do exhibit features different from classical minimal surfaces in that they may stick to the boundary.  See the work of Dipierro, Savin, and Valdinoci \cite{DSV17} and Sipierro and Valdinoci \cite{DV18}.  If $E$ is a $\sigma$-minimal surface, then Caffarelli, Roquejoffre, and Savin \cite{CRS10} showed that a pointwise condition must hold on $\partial E$.  This condition was used by Abatangelo and Valdinoci \cite{AV14} to define a nonlocal analog of mean curvature and directional curvature.  Motion by this nonlocal mean-curvature has been studied using level set methods.  See the work of Imbert \cite{I09} and Chambelle, Marini, and Ponsiglione \cite{CMP12,CMP13,CMP15}.

The motivation for the name of Per$_\sigma$, as well as the factor $\alpha_{n-1}^{-1}$ that appears in its definition, comes from the fact that in an appropriate limit, the fractional perimeter converges to the classical notion of perimeter.  Caffarelli and Valdinoci \cite{CV11} showed that if $E$ has smooth boundary, then
\beqn\label{limsPer}
\lim_{\sigma\uparrow 1}(1-\sigma)\text{Per}_\sigma(E,B_r)=\cH^{n-1}(\partial E\cap B_r)
\eeqn
for almost every $r>0$, where $B_r$ is the ball centered at the origin of radius $r$.  Ambrosio, Philippis, and Martinazi \cite{APM11} were able to show that the fractional perimeter $\Gamma$-converges to the classical notion of perimeter as $\sigma$ goes to $1$.  The asymptotics of the fractional perimeter as $\sigma$ goes to zero have also been studied.  Dipierro, Figalli, Palatucci, and Valdinoci \cite{DF13} proved that
\beqn
\lim_{\sigma\downarrow 0} \sigma\text{Per}_\sigma(E,\Omega)=\frac{1}{\alpha_{n-1}} [ (1-a(E))\cH^n(E\cap\Omega)+a(E)\cH^n(\Omega\setminus E)],
\eeqn
where $\displaystyle a(E)\coloneqq\lim_{\sigma\downarrow 0}\frac{\sigma}{\omega_{n-1}} \int_{E\setminus B_1} |y|^{-n-\sigma}dy$ and $\omega_{n-1}$ is the surface area of a unit ball in $\R^n$.

The fractional perimeter is a measure of area for a surface that is the boundary of a set.  Of course, many surfaces are not the boundary of a set.  Paroni, Podio-Guidugli, and Seguin \cite{PPGS99} noticed that when $\partial E$ is smooth, the fractional perimeter \eqref{sper} can be written as
\beqn
\text{Per}_\sigma(E,\Omega)=\frac{1}{2\alpha_{n-1}}\int_{\cX(\partial E)} \frac{\max\{\chi_\Omega(x),\chi_\Omega(y)\}}{|x-y|^{n+\sigma}} dxdy,
\eeqn
where $\cX(\partial E)$ denotes the set of pairs $(x,y)\in\R^n\times\R^n$ such that the line segment connecting $x$ and $y$ intersects $\partial E$ an odd number of times.  Since the right-hand side of the previous equation only depends on the set $E$ through the surface $\partial E$, this allows for the definition of a fractional notion of area.  Namely, if $\cS$ is a smooth surface, the fractional area of $\cS$ relative to $\Omega$ is defined by
\beqn\label{sarea}
\text{Area}_\sigma(\cS,\Omega)\coloneqq\frac{1}{2\alpha_{n-1}}\int_{\cX(\cS)} \frac{\max\{\chi_\Omega(x),\chi_\Omega(y)\}}{|x-y|^{n+\sigma}} dxdy.
\eeqn
Under the assumption that $\cS\subseteq\Omega$, the fractional area converges to the usual notion of area in the sense that
\beqn\label{limsarea}
\lim_{\sigma\uparrow 1}(1-\sigma)\text{Area}_\sigma(\cS,\Omega)=\cH^{n-1}(\cS).
\eeqn

In the work of Seguin \cite{S20,S20c} the above idea was used to define a fractional notion of length for curves in $\R^n$.  To see how this works, consider the case when $n=2$ so that $\cS$ is a curve in $\R^2$.  In that case,
\beqn\label{2dsarea}
\text{Area}_\sigma(\cS,\Omega)=\frac{1}{4} \int_{\cX(\cS)}  \frac{\max\{\chi_\Omega(x),\chi_\Omega(y)\}}{|x-y|^{n+\sigma}} dxdy.
\eeqn
This expression involves an integral over the set of line segments, described by their end points, that intersect $\cS$ an odd number of times.  Rather that using endpoints $x$ and $y$ to describe a given line segment, it can be viewed as a one-dimensional disk with a center $p$, a unit vector $\bu$ normal to the disk, and a radius $r$.  When this is done, these objects are related through
\beqn
(x,y)=(p-r\bu',p+r\bu'),
\eeqn
where $\bu'$ is obtained by rotating $\bu$ clockwise by 90$^\circ$.  Using this change of variables, \eqref{2dsarea} becomes
\beqn
\text{Area}_\sigma(\cS,\Omega)=\frac{1}{2} \int_{\cD(\cS)}  \frac{\max\{\chi_\Omega(p+r\bu'),\chi_\Omega(p-r\bu')\}}{(2r)^{1+\sigma}} d\cH^4(p,\bu,r),
\eeqn
where $\cD(\cS)$ consists of all one-dimensional disks that intersect $\cS$ an odd number of times.  As $\cS$ is one dimensional here, this suggests a fractional notion of length.  Namely, if $\cC$ is a smooth curve in $\R^n$, then its fractional length is given by
\beqn\label{slen}
\text{Len}_\sigma(\cC,\Omega)\coloneqq\int_{\cD(\cC)}  \frac{\sup\{\chi_\Omega(p+r\bv):\bv\in \cU(\{\bu\}^\perp)\}}{r^{1+\sigma}} d\cH^{2n}(p,\bu,r).
\eeqn
Here $\{\bu\}^\perp$ is the collection of vectors orthogonal to $\bu$, $\cU(\{\bu\}^\perp)$ is the set of unit vectors in $\{\bu\}^\perp$, and $\cD(\cC)$ is the collection of triples $(p,\bu,r)$ such that the corresponding disk
\beqn
D(p,\bu,r)\coloneqq\{p+\xi\bv:\bv\in\cU(\{\bu\}^\perp), \xi\in[0,r)\}
\eeqn
intersects the curve $\cC$ an odd number of times.  The presence of the supremum in \eqref{slen} is analogous to the maximum in \eqref{2dsarea}.  The interpretation of this term is that only those disks whose boundary intersects $\Omega$ contribute to the integral.  Seguin \cite{S20,S20c} established that under the assumption $\cC\subseteq\Omega$,
\beqn\label{limslen}
\lim_{\sigma\uparrow 1}(1-\sigma)\text{Len}_\sigma(\cC,\Omega)=\frac{4\pi^{n-1}}{\Gamma(\sfrac{n+1}{2})\Gamma(\sfrac{n-1}{2})(n-1)}\cH^1(\cC).
\eeqn

The definitions \eqref{sarea} and \eqref{slen} provide fractional notions of measure for $n-1$ and $1$ dimensional manifolds in $\R^n$.  The goal of this work is to bridge the gap between these results and define a fractional notion of measure for $k$-dimensional manifolds and establish a result analogous to \eqref{limsarea} and \eqref{limslen}.  To motivate the definition we introduce, begin by noticing that the fractional area involves considering a set of line segments, or one-dimensional disks, that intersect the surface on odd number of times, while the fractional length involved a set of $(n-1)$-dimensional disks. The analogue is to consider $(n-k)$-dimensional disks that intersect a given $k$-dimensional manifold an odd number of times. Notice that in every case, the dimension of the geometric object being measured plus the dimension of the disks involved adds up to $n$.  Moreover, the intersection of a $k$-dimensional manifold with a $(n-k)$-dimensional disk generally results in a collection of points, rather than higher dimensional sets.  Thus, it makes sense to ask whether the intersection consists of an even or odd number of points.  Similar to the definitions \eqref{sarea} and \eqref{slen}, the definition of the fractional $k$-dimensional measure will involve integrating over all disks that intersect the manifold an odd number of times.  

To define the disks being used, and perform various calculations, we make use of $k$-vectors.  For a detailed introduction to $k$-vectors, see Ros\'en \cite{R19}.  In an effort to make this work self contained, we have included an introduction to $k$-vectors in the Section~\ref{Sect2}.  Here we also set notion and mention some facts about the gamma and beta functions that will be used.  In Section~\ref{Sect3} we discuss the set of disks used to define the fractional $k$-dimensional measure and prove preliminary results about the dimensions of relevant subsets. In Section~\ref{Sect4} we introduce a fractional $k$-dimensional measure and prove it converges in an appropriate limit to the $\cH^k$ measure.  Section~\ref{Sect5} contains several directions of research that can be investigated involving the fractional measure.  Lastly, the Appendix contains the change of variable formula and other computations necessary for proving the results of Sections~\ref{Sect3} and \ref{Sect4}.

\section{Notation and mathematical preliminaries}\label{Sect2}

Let $\Real^+\coloneqq(0,\infty)$ denote the set of positive numbers not including zero and $\Real^+_0\coloneqq[0,\infty)$ the set of positive numbers including zero.  Given an inner-product space $\cV$, let $\cU(\cV)$ denote the set of all unit vectors in $\cV$.  If $\cZ$ is a subset of $\cV$, let $\cZ^\perp$ denote those vectors in $\cV$ that are orthogonal to every vector in $\cZ$.   

Recall that the gamma function $\Gamma$ is defined for all positive numbers $x$ by
\beqn\label{defgamma}
\Gamma(x)\coloneqq\int_0^\infty y^{x-1}e^{-y}dy
\eeqn
and can be viewed as a generalization of the factorial in that it satisfies
\beqn
\Gamma(x+1)=x\Gamma(x),\qquad x\in\Real^+.
\eeqn
It is known that
\beqn\label{gammahalf}
\Gamma(\sfrac{1}{2})=\sqrt{\pi}.
\eeqn
The beta function $B$ is closely related to the gamma function.  While there are many equivalent expressions, the one must useful in this work is
\begin{align}
\label{betatrig}B(x,y)&=2\int_0^{\pi/2} \sin^{2x-1}\theta\cos^{2y-1}\theta d\theta.
\end{align}
The beta and gamma functions are related through the identity
\beqn\label{betagamma}
B(x,y)=\frac{\Gamma(x)\Gamma(y)}{\Gamma(x+y)}.
\eeqn

The volume $\alpha_n$ and surface area $\omega_{n-1}$ of a unit ball in $\R^n$ can be given in terms of the gamma function:
\begin{align}
\label{volball}\alpha_n&=\frac{\pi^{n/2}}{\Gamma(1+n/2)},\\
\label{areaball}\omega_{n-1}&=\frac{2\pi^{n/2}}{\Gamma(n/2)}.
\end{align}
Let SO$(n)$ denote the set of all rotations of $\R^n$.  Recall that this is a $\sfrac{n(n-1)}{2}$-dimensional manifold.  Viewing $\SO(n)$ as a subset of $\R^{n^2}$ with the natural metric of this space, we have
\beqn\label{SOnmeas}
\cH^{\sfrac{n(n-1)}{2}}(\SO(n))=2^{\sfrac{n(n-1)}{4}}\prod_{j=1}^{n-1}\omega_j=\frac{2^{\sfrac{(n+4)(n-1)}{4}}\pi^{\sfrac{(n+2)(n-1)}{4}}}{\displaystyle\prod_{j=2}^n\Gamma(\sfrac{j}{2})}.
\eeqn
See, for example, Zhou and Shi \cite{SZ14}.

There are different ways to define multivectors.  Here, we take the approach of viewing them as skew multilinear mappings.  While this is not the definition used in the book by Ros\'en \cite{R19}, it is equivalent.  See Federer \cite{Fed}.  Fix a natural number $k$ such that $1\leq k\leq n$.  Given $k$ vectors $\bv_i$, $i\in\{1,\dots,k\}$, define the $k$-vector $\bv_1\wedge\dots\wedge\bv_k$ to be the multilinear map from the $k$-fold product $(\R^n)^k$ to $\R$ by
\beqn\label{simpkvector}
(\bv_1\wedge\dots\wedge\bv_k)(\ba_1,\dots,\ba_k)\coloneqq\text{det}\ 
\begin{blockarray}{cccc}
\begin{block}{(cccc)}
\bv_1\cdot\ba_1&\bv_1\cdot\ba_2&...&\bv_1\cdot\ba_k\\
\bv_2\cdot\ba_1&\bv_2\cdot\ba_2&...&\bv_2\cdot\ba_k\\
\vdots&\vdots&\ddots&\vdots\\
\bv_{k}\cdot\ba_1&\bv_{k}\cdot\ba_2&...&\bv_{k}\cdot\ba_k\\
\end{block}
\end{blockarray}\, ,
\qquad (\ba_1,\dots,\ba_k)\in(\R^n)^k.
\eeqn
Notice that this mapping is skew in the sense that switching any pair of $\ba_i$ results in changing the sign of the output.  The set of all $k$-vectors, denoted by $\Lambda^k(\R^n)$, is the linear span of all linear mappings of the form $\bv_1\wedge\dots\wedge\bv_k$.  We have $\text{dim}\,\Lambda^k(\R^n)={n\choose k}=\frac{n!}{k!(n-k)!}$. 

The collection of $k$-vectors of the form \eqref{simpkvector} are called simple $k$-vectors, and are denoted by $\hat\Lambda^k(\Real^n)$.  While $\Lambda^k(\R^n)$ is a vector space, $\hat\Lambda^k(\R^n)$ is not.  However, it is a manifold.  Moreover, simple $k$-vectors span the set of all $k$-vectors.  For a simple $k$-vector $\bomega=\bv_1\wedge\dots\wedge\bv_k$, we denote by $[\bomega]$ the subspace of $\Real^n$ spanned by the $\bv_i$.  Let $\bP_{\bomega}$ and $\bP_{\bomega}^\perp$ be the functions that project elements of $\R^n$ into $[\bomega]$ and $[\bomega]^\perp$, respectively, so that $\bP_{\bomega}+\bP^\perp_{\bomega} =\textbf{1}_{n}$.

We define several linear mappings on $k$-vectors and pairs of $k$-vectors.  To do so, we first define them on simple $k$-vectors and then extend them to all $k$-vectors by linearity.  The reason that this process works is because $k$-vectors satisfy a universality property.  See Ros\'en \cite{R19} for the details. As a first example of this, consider a linear mapping $\bL:\Real^n\rightarrow \Real^n$.  This induces a linear mapping from $\Lambda^k(\Real^n)$ to itself that we will also denote by $\bL$ such that it acts on a simple $k$-vector $\bu_1\wedge\dots\wedge\bu_k\in\hat\Lambda^k(\R^n)$ by
\beqn
\bL(\bu_1\wedge\dots\wedge\bu_k)\coloneqq\bL\bu_1\wedge\dots\wedge\bL\bu_k.
\eeqn

Given two simple $k$-vectors $\bu_1\wedge\dots\wedge\bu_k,\bv_1\wedge\dots\wedge\bv_k\in\hat\Lambda^k(\R^n)$, define their inner-product by
\beqn\label{kvip}
(\bu_1\wedge\dots\wedge\bu_k)\cdot(\bv_1\wedge\dots\wedge\bv_k)\coloneqq\text{det}\ 
\begin{blockarray}{cccc}
\begin{block}{(cccc)}
\bu_1\cdot\bv_1&\bu_1\cdot\bv_2&...&\bu_1\cdot\bv_k\\
\bu_2\cdot\bv_1&\bu_2\cdot\bv_2&...&\bu_2\cdot\bv_k\\
\vdots&\vdots&\ddots&\vdots\\
\bu_{k}\cdot\bv_1&\bu_{k}\cdot\bv_2&...&\bu_{k}\cdot\bv_k\\
\end{block}
\end{blockarray}\, .
\eeqn
Orthogonal simple $k$-vectors have the following geometric interpretation: given $\bomega,\bnu\in\hat\Lambda^k(\R^n)$ such that $\bomega\cdot\bnu=0$, then there is an $\ba\in[\bomega]$ such that $\ba\in[\bnu]^\perp$.  

We will denote the set of simple $k$-vectors of unit length by $\hat\Lambda^k_u(\R^n)\coloneqq\cU(\hat\Lambda^k(\Real^n))$.  Elements of $\hat\Lambda^k_u(\R^n)$ can be viewed as oriented $k$-dimensional subspaces of $\R^n$.  Thus, $\hat\Lambda^k_u(\R^n)$ can be used as a double cover for the Grassmannian $\textbf{Gr}(k,\R^n)$.  It follows that $\text{dim}\,\hat\Lambda^k_u(\R^n)=k(n-k)$.  Given a subspace $\cV$ of $\Real^n$, we will use the notation $\hat\Lambda_u^k(\cV)$ for those unit simple $k$-vectors $\bomega$ such that $[\bomega]\subseteq\cV$.

The exterior product $\wedge:\Lambda^k(\R^n)\times\Lambda^l(\R^n)\rightarrow\Lambda^{k+l}(\R^n)$ is the bilinear map such that given simple vectors $\bu_1\wedge\dots\wedge\bu_k\in\hat\Lambda^k(\R^n)$ and $\bv_1\wedge\dots\wedge\bv_l\in\hat\Lambda^l(\R^n)$, we have
\beqn
(\bu_1\wedge\dots\wedge\bu_k)\wedge(\bv_1\wedge\dots\wedge\bv_l)\coloneqq\bu_1\wedge\dots\wedge\bu_k\wedge\bv_1\wedge\dots\wedge\bv_l.
\eeqn
We will also make use of the interior product of multivectors.  For $k\geq l$, the (left) interior product $\lrcorner:\Lambda^l(\R^n)\times\Lambda^k(\R^n)\rightarrow\Lambda^{k-l}(\R^n)$ is the bilinear map such that for $\bnu\in\Lambda^l(\R^n)$ and $\blambda\in\Lambda^k(\R^n)$, $\bnu\lrcorner\blambda$ is the unique $(k-l)$-vector satisfying
\beqn \label{movingdot}
\bomega\cdot(\bnu\lrcorner\blambda)=(\bnu\wedge\bomega)\cdot\blambda,\qquad \bomega\in\Lambda^{k-l}(\R^n).
\eeqn

A useful identity involving the interior product is
\beqn\label{moveuk}
(\bomega\wedge\blambda)\lrcorner\bmu=\blambda\lrcorner(\bomega\lrcorner\bmu)
\eeqn
for all $\bomega\in\Lambda^{k}(\R^n)$, $\blambda\in\Lambda^{l}(\R^n)$, and $\bmu\in\Lambda^{m}(\R^n)$ for $k+l\leq m$.  There is also a varient on Lagrange's identity: if $\ba\in\Real^n$ and $\bomega\in\Lambda^k(\Real^n)$, then
\beqn\label{Lagrange}
|\ba|^2|\bomega|^2=|\ba\lrcorner\bomega|^2+|\ba\wedge\bomega|^2.
\eeqn
Another useful identity is the anticommutation relation
\beqn\label{anticom}
\ba\lrcorner(\bb\wedge\bomega)+\bb\wedge(\ba\lrcorner\bomega)=(\ba\cdot\bb)\bomega,
\eeqn
which holds for all $\ba,\bb\in\Real^n$ and $\bomega\in\Lambda^k(\Real^n)$.  It can be shown that for a simple $k$-vector $\bomega$,
\beqn
[\bomega]=\{\ba\in\R^n\ :\ \ba\wedge\bomega=\textbf{0}\}\quad\text{and}\quad [\bomega]^\perp=\{\ba\in\R^n\ :\ \ba\lrcorner\bomega=\textbf{0}\}.
\eeqn

\begin{lemma}\label{projflip}
If $\bmu\in\hat\Lambda^k(\R^n)$ and $\aa\in\R^n$, then 
\beqn\label{kvectorid}
|\aa\lrcorner\bmu|=|\bmu||\bP_{\bmu}\aa|\qquad\mbox{and}\qquad |\ba||\bP_\aa^\perp\bmu|=|\bP_{\bmu}^\perp \aa|.
\eeqn
\end{lemma}
\begin{proof}
We begin by proving \eqref{kvectorid}$_1$ when the $k$-vector and vector have unit magnitude.  Let $\bmu\in\hat\Lambda_u^k(\R^n)$ and $\aa\in\cU(\R^n)$.  Since $\bmu\in \hat\Lambda_u^k(\R^n)$, we can write 
$\bmu=\tt_1\wedge...\wedge\tt_k$, where $\{\tt_1,...,\tt_k\}$ is an orthonormal basis for $[\bmu]$. Expanding $|\aa\lrcorner\bmu|^2$ we have 
\[
|\aa\lrcorner\bmu|^2=|(\aa\cdot\tt_1)(\tt_2\wedge...\wedge\tt_k)-(\aa\cdot\tt_2)(\tt_1\wedge \tt_3\wedge...\wedge\tt_k)+...+(-1)^{k+1}(\aa\cdot \tt_{k})(\tt_1\wedge....\wedge\tt_{k-1})|^2.
\]
Because each of the terms in the sum are orthogonal and of unit length, it follows that
\begin{align}
|\aa\lrcorner\bmu|^2&=|(\aa\cdot\tt_1)(\tt_2\wedge...\wedge\tt_k)|^2+|(\aa\cdot\tt_2)(\tt_1\wedge \tt_3\wedge...\wedge\tt_k)|^2+...+|(\aa\cdot \tt_{k})(\tt_1\wedge....\wedge\tt_{k-1})|^2\\
&=|(\aa\cdot\tt_1)|^2+\dots+|(\aa\cdot\tt_k)|^2 =|\bP_{\bmu}\aa|^2.\label{innerprojbit}
\end{align}
To generalize the result to $\bmu$ and $\ba$ without unit magnitude, notice that
\beqn\label{scaling}
|\ba\lrcorner\bmu|=|\bmu||\ba| \Big |\frac{\ba}{|\ba|}\lrcorner \frac{\bmu}{|\bmu|}\Big|=|\bmu| |\ba||\bP_{\bmu}\frac{\ba}{|\ba|}|=|\bmu||\bP_{\bmu}\ba|.
\eeqn

To establish \eqref{kvectorid}$_2$, we first represent the projection onto the space perpendicular to $\ba$ by $\bP_\aa^\perp\bmu=\aa\lrcorner(\aa\wedge \bmu)$; see\cite{R19}.  So by \eqref{movingdot}, we have 
\beqn\label{firstbit}
|\bP_\aa^\perp\bmu|^2=\aa\lrcorner(\aa\wedge \bmu)\cdot \aa\lrcorner(\aa\wedge \bmu)=\aa\wedge(\aa\lrcorner(\aa\wedge \bmu))\cdot \aa\wedge \bmu.
\eeqn
By \eqref{anticom}, the fact that $\aa\wedge\aa\wedge\bmu=\bzero$, and because $\aa$ is a unit vector, it follows that 
\beqn\label{nextbit}
\aa\wedge(\aa\lrcorner(\aa\wedge \bmu))\cdot \aa\lrcorner \bmu=\big[|\aa|^2\aa\wedge \bmu-\aa\lrcorner (\aa\wedge\aa\wedge\bmu) \big]\cdot \aa\lrcorner \bmu=| \aa\wedge \bmu|^2.
\eeqn
Next, \eqref{Lagrange} implies
\beqn\label{lastbit}
| \aa\wedge \bmu|^2=|\aa|^2|\bmu|^2-|\aa\lrcorner\bmu|^2=1-|\aa\lrcorner\bmu|^2.
\eeqn
Combining \eqref{firstbit}--\eqref{lastbit} and using \eqref{kvectorid}$_1$ we have 
\[
|\bP_{\aa}^\perp\bmu|^2=1-|\bP_{\bmu}\aa|^2=|\bP_{\bmu}^\perp\aa|^2.
\]
Generalizing the result to when $\ba$ and $\bmu$ do not have unit magnitude follows from a scaling argument similar to that in \eqref{scaling}.
\end{proof}

We will illustrate the interior product by mentioning a fact that will be of use later.  Let $\bomega\in\hat\Lambda_u^4(\Real^n)$.  It follows that there are orthonormal vectors $\bu_1,\bu_2,\bu_3$ and $\bu_4$ such that $\bomega=\bu_1\wedge\bu_2\wedge\bu_3\wedge\bu_4$.  From the definitions above, for $\ba\in\Real^n$
\beqn
\ba\lrcorner \bomega= (\ba\cdot\bu_1)\bu_2\wedge\bu_3\wedge\bu_4-(\ba\cdot\bu_2)\bu_1\wedge\bu_3\wedge\bu_4+(\ba\cdot\bu_3)\bu_1\wedge\bu_2\wedge\bu_4-(\ba\cdot\bu_4)\bu_1\wedge\bu_2\wedge\bu_3.
\eeqn
Thus,
\beqn
\bu_2\lrcorner \bomega= -\bu_1\wedge\bu_3\wedge\bu_4,
\eeqn
and for $\bb\in\Real^n$,
\beqn
\bb\wedge(\bu_2\lrcorner\bomega)=\bu_1\wedge\bb\wedge\bu_3\wedge\bu_4.
\eeqn
This kind of identity generalizes.  Namely, if $\bomega=\bu_1\wedge\dots\wedge\bu_k\in\hat\Lambda_u^k(\Real^n)$ and $\bb\in\Real^n$, then $\bb\wedge(\bu_i\lrcorner\bomega)$ is the $k$-vector obtained by replacing $\bu_i$ with $\bb$ in $\bomega$---that is,
\beqn\label{replacewedgeID}
\bb\wedge(\bu_i\lrcorner\bomega)=\bu_1\wedge\dots\wedge\bu_{i-1}\wedge\bb\wedge\bu_{i+1}\wedge\dots\bu_k.
\eeqn

\smallskip
\section{Set of disks}\label{Sect3}

In this section we discuss the set of $(n-k)$-dimensional disks in $\Real^n$.  These will be used to define a fractional $k$-dimensional measure.  Several measure theoretic properties of these disks will be established that clarify which subset of disks is being integrated over when defining the fractional measure.  We find it convenient to work with oriented disks, so for the rest of the paper a disk refers to an oriented disk.

Each $(n-k)$-dimensional disk can be represented by its center $p$, radius $r>0$, and perpendicular $k$-dimensional (oriented) subspace represented by $\bom\in \hat{\Lambda}^k_u(\R^{n})$ so that
\beqn\label{diskdef}
D(p,\bom, r)\coloneqq\{p+\xi \mathbf{v}\in\R^n:\mathbf{v}\in\cU([\bomega]^\perp), \xi\in [0,r)\}
\eeqn
is such a disk.  By the boundary $\pa D(p,\bom, r)$ of one of these disks we mean the $(n-2)$-dimensional manifold 
\[
\pa D(p,\bom, r)\coloneqq\{p+r \mathbf{v}\in\R^n:\mathbf{v}\in\cU([\bomega]^\perp)\}.
\]
Thus, the set of all disks can be described by the set
\[
\D\coloneqq\R^n\times \hat{\Lambda}^k_u(\R^{n})\times \R^+.
\]
Note that this set has Hausdorff dimension $n+k(n-k)+1$.

Fix a bounded $k$-dimensional manifold $\cM$ such that its boundary $\partial\cM$ is a $(k-1)$-dimensional manifold and, thus, $\overline\cM$ is a manifold with boundary. Let $\textbf{vol}_\cM$ denote a volume form for $\overline\cM$, so that for each $z\in\overline\cM$, $\textbf{vol}_\cM(z)\in\hat\Lambda^k_u(\Real^n)$ and $[\textbf{vol}_\cM]=T_z\overline\cM$.  Keep in mind that we are not assuming that $\cM$ is orientable, so it is possible that $\textbf{vol}_\cM$ is not continuous.  We define the following subsets of $\D$: 
\begin{align*}
\D_{\pa \M}&\coloneqq\{(p,\bom,r)\in\D:\H^0(\overline{D}(p,\bom,r)\cap \pa \M)\neq 0 \},\\
\D_{\tan}&\coloneqq\{(p,\bom,r)\in \D: \text{there is } z\in \overline{D}(p,\bom,r)\cap \overline\M\text{ such that } \bomega\cdot\textbf{vol}_\cM(z)=0\},\\
\D_\infty&\coloneqq\{(p,\bom,r)\in \D: \H^0(\overline{D}(p,\bom,r)\cap \M)=\infty\},\\ 
\D_{\pa D}&\coloneqq\{(p,\bom,r)\in \D: \H^0(\pa D(p,\bom,r)\cap \M)\geq1\}. 
\end{align*}
Roughly speaking, the interpretation of these sets is as follows: $\cD_{\partial \cM}$ consists of those disks that intersect the boundary of $\cM$, $\cD_\text{tan}$ consists of those disks that are tangent to the manifold, $\cD_\infty$ is the collection of disks that intersect $\cM$ an infinite number of times, and $\cD_{\partial D}$ consists of those disks whose boundary intersects $\cM$.

In the following lemma we show that all of these sets have at most dimension $n+k(n-k)$, one less than the dimension of $\cD$. To do so we will use the set
\beqn\label{Wdef}
\W\coloneqq\{(\mathbf{a},\bom)\in \cU(\R^n)\times \hat{\Lambda}^k_u(\R^{n}): \mathbf{a}\in [\bom]^\perp\}.
\eeqn
Each element in $\W$ gives a direction $\mathbf{a}$ and a unit magnitude simple $k$-vector $\bom$, which gives a subspace that is perpendicular to $\mathbf{a}$.  Put another way, $\bom$ is a $k$-vector in the $(n-1)$-dimensional space $[\ba]^\perp$.  Therefore, the dimension of $\W$ is $n-1+k(n-k-1)$.

\begin{lemma}\label{lemmeasure}
Given an open bounded set $\cE\subseteq\cD$, the following results hold:
\begin{enumerate}
\item $\H^{n+k(n-k)}(\cD_{\pa \M}\cap\E)<\infty$,
\item $\H^{n+k(n-k)}(\cD_{\rm{tan}}\cap\E)<\infty$,\label{tan}
\item$\H^{n+k(n-k)}(\cD_{\pa D}\cap\E)<\infty$,\label{paD} 
\item$\D_\infty\subseteq\D_{\text{\rm tan}}$.
\end{enumerate}
\end{lemma}
\begin{proof}
For the proof of the first part of this lemma we will make use of the function $\Xi$ defined in \eqref{xi} of the Appendix.  Set $\E_\Xi=\Xi^{-1}(\E)$ and choose $R>0$ so that, if $(p,\bom,r)\in \E$, then  $r\in (0,R]$.

\textit{Item 1.} Consider the set
\[
A_{\pa \M}\coloneqq \pa \M\times \W \times \R^+\times \R^+.
\]
Notice that $A_{\pa \M}$ has Hausdorff dimension $n+k(n-k)$ and $\D_{\pa \M}\cap \E\subseteq \Xi(A_{\pa \M}\cap \E_{\Xi})$. Since $\Xi$ is Lipschitz on $A_{\pa \M}$, having $\H^{n+k(n-k)}(A_{\pa\M}\cap \E_\Xi)<\infty$ implies $\H^{n+k(n-k)}(\D_{\pa\M}\cap \E)<\infty$.

\textit{Item 2.} 
Consider the set
\begin{align*}
A_{\text{tan}}\coloneqq\bigcup_{z\in \M} \big(&\{z\}\times \{(\mathbf{a},\bom)\in\W: \bomega\cdot\textbf{vol}_\cM(z)=0\}\big)\times\{ (\xi,r)\in\R^+\times \R^+:0\le\xi\le r\le R\}.
\end{align*}
The argument in this case is similar to that as in Item 1 with $A_{\partial\cM}$ replaced by $A_\text{tan}$, but we need to calculate the dimension of $A_{\text{tan}}$.  For $z$ there are $k$ dimensions, $(\ba,\bomega)$ there are $n-1+k(n-k-1)$ dimensions, but you have to subtract one because of the constraint $\bomega\cdot\textbf{vol}_\cM(z)=0$, and $\xi$ and $r$ consist of two more dimensions.  Adding these up one finds that the Hausdorff dimension of $A_{\text{tan}}$ is $n+k(n-k)$.

\textit{Item 3.} To describe the disks in $\D_{\pa D}$, we define a map 
\[
\Psi: \M\times \W \times \R^+\rightarrow \R^n\times \hat{\Lambda}^k_u(\R^n) \times\R^+\]
 by
\begin{equation}\label{Psi}
\Psi(z,\textbf{a},\bom,r)\coloneqq(z+r\textbf{a},\bom,r), \qquad (z,\textbf{a},\bom,r)\in \M\times\W\times \R^+.
\end{equation}
Set $\cE_\Psi\coloneqq\Psi^{-1}(\cE)$ and consider
\[
A_{\pa D}\coloneqq \M\times \W \times \R^+.
\]
As in the previous parts, one can see that $\D_{\pa D}\cap \E\subseteq \Psi(A_{\pa D}\cap \E_\Psi)$. Since $\Psi$ is Lipschitz and $\H^{n+k(n-k)}(A_{\pa D}\cap \E_\Psi)<\infty,$ it follows that $\H^{n+k(n-k) }(\D_{\pa D} \cap\E)<\infty.$

\textit{Item 4.}  
 Consider $(p,\bom,r)\in \D_\infty$, so that $\H^0(\overline{D}(p,\bom,r)\cap \M)=\infty.$ Because $\overline{D}(p,\bom,r)\cap \overline{\M}$ is compact it follows that it has an accumulation point, which we will call $z\in \overline{D}(p,\bom,r)\cap \overline{\M}$. Suppose that $(p,\bom,r)\notin\D_{\tan},$ so that $\bomega\cdot \textbf{vol}_\cM(z)\not=0$.  This means that $T_z\overline\cM$ and $[\bomega]^\perp$ have no directions in common.  In a small enough neighborhood of $z$ we can locally approximate $\overline\M$ with its tangent plane. Since no direction in the tangent plane matches the directions that span the disk, there are no points in a small neighborhood of $z$ in $\overline{D}(p,\bom,r)\cap \overline\M$ other than $z$, which contradicts $z$ being an accumulation point.  Thus, we must have $(p,\bom,r)\in\D_{\tan}$.

\end{proof}

We also need the set of disks
\beqn
\cD(\cM)\coloneqq\{(p,\bomega,r)\in\cD : \cH^0(D(p,\bomega,r)\cap\cM)\ \text{is odd}\},
\eeqn
which consists of those disks that intersect $\cM$ an odd number of times.  To define the fractional notion of measure, we will be integrating over this set of disks, which has Hausdorff dimension $n+k(n-k)+1$. It follows from Lemma~\ref{lemmeasure} that when counting the intersections between a disk and $\cM$ it does not matter whether you include intersections that occur on the boundary of $\cM$ or the boundary of the disk.  Moreover, when counting intersections you can ignore points at which the surface is tangent to the disk.

\section{Fractional measure} \label{Sect4}

In this section we define a fractional notion of $k$-dimensional measure and show that, in an appropriate limit, it converges to the $\H^k$-measure up to a multiplicative constant. 

Let $\Om$ be an open, bounded set. We define the $k$-dimensional $\s$-measure of $\M$ relative to $\Om$ by
\begin{equation}\label{smeasdef}
\text{Meas}_{\s}^k(\M,\Om)\coloneqq\int_{\D(\M)}r^{k-n-\s}  \sup_{\bu\in \cU([\bomega]^\perp)}\chi_{\Om}(p+r\mathbf{u})\,d \H^{n+k(n-k)+1}(p,\bom,r).
\end{equation}
For simplicity suppose that $\M\subseteq\Om.$ Defining
\[
\D_\Om(\M)\coloneqq\{ (p,\bom,r)\in \D(\M):\pa D(p,\bom,r)\cap\Om\neq \emptyset)\},
\]
we have
\begin{equation}\label{area formula}
\text{Meas}_{\s}^k(\M,\Om)=\int_{\D_\Om(\M)}r^{k-n-\s} \,d \H^{n+k(n-k)+1}(p,\bom,r).
\end{equation}

First we verify that this integral is finite. In the following we let $\diam(\Om)$ denote the diameter of $\Om;$ this is finite because $\Om$ is bounded. Using the function defined in \eqref{xi}, it turns out that
\begin{equation}\label{oddballs}
\D_\Om(\M)\subseteq \Xi\Big(\bigcup_{\xi\in \R^+_0}\M\times \W\times\{\xi\}\times [\xi,\xi+\diam(\Om)]\Big).
\end{equation}
Indeed, if $(p,\bom,r)\in \D_\Om(\M)$, then $D(p,\bom,r)$ intersects $\M$ a finite number of times, so we can find $z\in \M\cap D(p,\bom,r)$ with minimum distance to $p$ such that $(p-z)\in[\bom]^\perp$. Set $\xi=|p-z|$ and $\mathbf{a}=(p-z)/\xi$ if $z\not =p$ and $\ba$ can be arbitrary if $z=p$.  It follows that $\Xi(z,\mathbf{a},\bom, \xi, r)=(p,\bom,r).$ Moreover, since $z\in D(p,\bom,r)$ we know $\xi\le r,$ and $\pa D(p,\bom,r)\cap\Om\neq \emptyset$ implies $r\le \xi+\diam(\Om)$. To verify this last claim, let $q\in\partial D(p,\bomega,r)\cap\Omega$ and notice that
\beqn
r=|p-q|\leq |p-z|+|z-q|\leq\xi+\diam(\Omega)
\eeqn
since $z\in\M\subseteq \Omega$. So $(p,\bom,r)\in \M\times \W\times \cup_{\xi\in \R^+_0}(\{\xi\}\times [\xi,\xi+\diam(\Om)])$ and, thus,\eqref{oddballs} holds. 

We can now utilize the change of variables from Lemma \ref{COV} to confirm that the $\s$-measure is finite.  To begin with, setting $m=\text{dim}\,\cW$,
\begin{align*}
\text{Meas}_{\s}^k(\M,\Om)\le&\int_\M\int_0^\infty\int_{\W}\int_\xi^{\xi+\diam(\Om)} r^{k-n-\s}  
 \frac{\xi^{n-k-1}}{2^{k/2}}|\textbf{vol}_{\M}(z)\cdot \bomega|
d r d\H^{m}(\mathbf{a},\bom)d \xi dz\\
\le& \H^k(\M)2^{-k/2}\H^m(\W) \int_0^\infty\int_\xi^{\xi+\diam(\Om)} r^{k-n-\s}\xi^{n-k-1} drd\xi\\
=&\frac{\H^k(\M)2^{-k/2}\H^m(\W)}{n-k+\s-1}\int_0^\infty \xi^{-\s}- \frac{\xi^{n-k-1}}{(\xi+\diam(\Om))^{n-k+\s-1}} d \xi.
\end{align*}
To see why this integral is finite, fix $\e\in(0,\infty)$. We decompose the integral as 
\begin{align*}
\int_0^\infty \xi^{-\s}- \frac{\xi^{n-k-1}}{(\xi+\diam(\Om))^{n-k+\s-1}} d\xi=&\int_0^\e \xi^{-\s}- \frac{\xi^{n-k-1}}{(\xi+\diam(\Om))^{n-k+\s-1}} d\xi\\
&+\int_\e^\infty \xi^{-\s}- \frac{\xi^{n-k-1}}{(\xi+\diam(\Om))^{n-k+\s-1}}  d \xi.
\end{align*}
The first integral is finite because $\s\in(0,1)$ and the second term in the integrand is bounded. For the latter integral we apply the mean value theorem to 
\[
h_{\xi}(x)\coloneqq\xi^{n-k-1}(\xi+x)^{1+k-n-\s},\qquad x \in(0,\diam(\Om)).
\]
For each $\xi\in(\e,\infty)$ there is $c_\xi\in(0,\diam(\Om))$ such that 
\begin{align*}
h_{\xi}(0)-h_{\xi}(\diam(\Om))&=h_{\xi}'(c_\xi)(0-\diam(\Om))\\
&=\diam(\Om)(n-k-1+\s)\xi^{n-k-1}(\xi+c_\xi)^{k-n-\s}\\
&\le \diam(\Om)(n-k-1+\s)\xi^{-\s-1}.
\end{align*}
So 
\begin{align}
\int_\e^\infty \xi^{-\s}- \frac{\xi^{n-k-1}}{(\xi+\diam(\Om))^{n-k+\s-1}}d\xi &\le \int_\e^\infty \diam(\Om)(n-k-1+\s)\xi^{-\s-1}d\xi\nonumber\\
&=\frac{\diam(\Om)(n-k-1+\s)\e^{-\s}}{\s}.\label{xibound}
\end{align}
Therefore, the $\s$-measure is finite. 

\begin{remark}\label{keq0}
 {\rm The definition of $\text{Meas}^k_\sigma$ makes sense not only when $1\leq k<n$, but also in the case $k=0$ if interpreted properly.  By definition, the set of $0$-vectors are numbers, so $\hat\Lambda_u^0(\Real^n)=\{-1,1\}$.  Moreover, also by convention, $[\pm 1]=\{\textbf{0}\}$, and so $[\pm 1]^\perp=\Real^n$.  Thus, an $n$-dimensional disk as defined in \eqref{diskdef} is an oriented ball, the orientation depending on the sign of $\bomega\in\hat\Lambda_u^0(\Real^n)=\{-1,1\}$.  

The $\sigma$-measure defined in \eqref{smeasdef} is not useful when $k=n$.  In this case, $\hat\Lambda_u^n(\Real^n)=\{-\textbf{vol}_n,\textbf{vol}_n\}$, where $\textbf{vol}_n$ is a volume form for $\Real^n$.  Since $\cU([\textbf{vol}_n]^\perp)=\cU(\{\textbf{0}\})=\emptyset$, each disk in \eqref{diskdef} is the empty set.  Thus, there are no disks that intersect $\cM$ an odd number of times and, so, the integral in \eqref{smeasdef} is always zero.}
\end{remark}

Our next goal is to show that the fractional $k$-dimensional measure converges in an appropriate limit to the $k$-dimensional Hausdorff measure up to a multiplicative constant.  To calculate this constant, we use two preliminary results, both of which will use the set of $k$ orthonormal vectors that live in a subspace $\cX$ of $\Real^n$:
\beqn\label{Ukdef}
\mathcal{U}^k_\perp(\cX)\coloneqq\{(\uu_1,...,\uu_k):\uu_i\in\cU(\cX), \uu_i\cdot \uu_j=0, 1\leq i<j\leq k\}.
\eeqn

\begin{lemma}\label{itlemma}
Let $\cZ$ be a $q$-dimensional subspace of $\Real^n$, $\bmu\in\hat\Lambda^k(\cZ)$, and $p\in\Nat$ such that $p\leq k$ and $p\leq q$.  It follows that
\begin{multline}\label{itlemmaeq}
\int_{\cU^p_\perp(\cZ)}|(\bu_1\wedge\dots\wedge\bu_p)\lrcorner \bmu| d\cH^{\sfrac{p(2q-p-1)}{2}}(\bu_1,\dots,\bu_p)\\
=\frac{2^{\sfrac{p+1}{2}}\pi^{\sfrac{q-p+1}{2}}\Gamma(\sfrac{k-p+2}{2})}{\Gamma(\sfrac{k-p+1}{2})\Gamma(\sfrac{q-p+2}{2})} \int_{\cU^{p-1}_\perp(\cZ)}|(\bu_1\wedge\dots\wedge\bu_{p-1})\lrcorner \bmu| d\cH^{\sfrac{(p-1)(2q-p)}{2}}(\bu_1,\dots,\bu_{p-1}).
\end{multline}
\end{lemma}

\begin{proof}
The idea is to write the integral $\cU_\perp^p(\cZ)$ as an iterated integral using the coarea formula and then compute the inner integral, which is over a set of unit vectors, using Lemma~\ref{subspint}.  Towards this end, define $H:\cU^p_\perp(\cZ)\rightarrow \cU(\cZ)$ by $H(\uu_1,...,\uu_p)= \uu_p$.  Consider a fixed $(\bu_1,\dots,\bu_p)\in\cU^p_\perp(\cZ)$ and a collection of vectors $\{\be_1,\dots,\be_{q-p}\}$ such that $\{\bu_1,\dots,\bu_p,\be_1,\dots,\be_{q-p}\}$ forms an orthonormal basis for $\cZ$.  It follows that an orthonormal basis for $T_{(\uu_1,...,\uu_p)}\mathcal{U}^p_{\perp}(\cZ)\subseteq\cZ^p$ is the union of the following two sets:

\begin{enumerate}
\item $\{\bE_{ij}: 1\le i\le p, 1\le j\le q-p\},$ where $\bE_{ij}\coloneqq(\mathbf{0},...,\ee_j,...,\mathbf{0})$ is a list of $p$ vectors with $\ee_j$ in the $i$th position. This set contains $p(q-p)$ vectors.
\item $\{\bF_{ij}: 1\le i<j\le p\}$, where $\bF_{ij}\coloneqq\frac{1}{\sqrt{2}}(\textbf{0},..,\bu_j,...,-\bu_i,...,\textbf{0})$ is a list of $p$ vectors with $\uu_j$ in the $i$th position and $-\uu_i$ in the $j$th position.  This set contains $p(p-1)/2$ vectors.
\end{enumerate}
An orthonormal basis for $T_{\bu_p}\cU(\cZ)$ is $\{\bu_1,\dots,\bu_{p-1},\be_1,\dots,\be_{q-p}\}$.  Notice that
\begin{enumerate}

\item $\nabla H(\bu_1,\dots,\bu_p)\bE_{ij}=
\begin{cases}
\bzero & \text{if}\ i\not = p\\
\be_j & \text{if}\ i=p
\end{cases}\quad$ for $1\leq j\leq q-p$,

\item $\nabla H(\bu_1,\dots,\bu_p)\bF_{ij}=
\begin{cases}
\bzero & \text{if}\ j\not = p,\\
-\sfrac{1}{\sqrt{2}}\bu_i & \text{if}\ j=p
\end{cases}\quad$ for $1\leq i\leq p-1$.

\end{enumerate}
It follows that the matrix $[\nabla H]$ of $\nabla H$ relative to these basis is
\[
[\nabla H]= \begin{blockarray}{ccccc}
\bE_{ij} & \bE_{ij} & \bF_{ij} & \bF_{ij} &\\
i\not=j& i=p & j\not=p &j=p&\\
\begin{block}{(cccc)c}
\bzero&\bzero& \bzero&-\sfrac{1}{\sqrt{2}}\textbf{1}_{p-1} &p-1\\
\bzero&\bone_{q-p}&\bzero&\bzero&q-p\\
  \end{block}
\end{blockarray}
\]
and, hence,
\beqn
\det(\nabla H\nabla H^T)=2^{1-p}.
\eeqn
Thus, the coarea formula combined with \eqref{moveuk} yields
 \begin{align}\label{outside}
\nonumber&\int_{\cU^p_\perp(\cZ)}|(\bu_1\wedge\dots\wedge\bu_p)\lrcorner \bmu| d\cH^{\sfrac{p(2q-p-1)}{2}}(\bu_1,\dots,\bu_p)=\\
&2^{\frac{p-1}2}\int_{\mathcal{U}_\perp^{p-1}(\cZ)}\int_{\cU(\cZ\cap\{\uu_1,...,\uu_{p-1}\}^\perp)}|\uu_p\lrcorner ((\uu_1\wedge...\wedge \uu_{p-1})\lrcorner \bmu)|\, d \bu_pd\H^{\sfrac{(p-1)(2q-p)}{2}}(\uu_1,...,\uu_{p-1}).
\end{align}
Setting $\cX=\cZ\cap\{\bu_1,\dots\bu_{p-1}\}^\perp$ and $\blambda=(\uu_1\wedge...\wedge \uu_{p-1})\lrcorner \bmu$, the inner integral on the right-hand side of \eqref{outside} can be written as
\beqn
\int_{\cU(\cX)}|\uu_p\lrcorner \blambda|\, d \bu_p.
\eeqn
To evaluate this integral, we use Lemma~\ref{subspint} with $\cY=[\blambda]$.  Making use of \eqref{kvectorid}$_1$, setting $d=\text{dim}(\cX)$ and $\ell=\text{dim}(\cX\cap[\blambda])$, and using \eqref{betatrig} and \eqref{areaball} we obtain
\begin{align}
\int_{\cU(\cX)}|\uu_p\lrcorner \blambda|\, d \bu_p&=\int_{\cU(\cX\cap[\blambda])}\int_{\cU(\cX\cap[\blambda]^\perp)}\int_0^{\pi/2}|\blambda|\cos\theta\cos^{\ell-1}\theta\sin^{d-\ell-1}\theta d\theta d\by' d\by\\
&=\omega_{\ell-1}\omega_{d-\ell-1}|\blambda|\sfrac{1}{2}B(\sfrac{d-\ell}{2},\sfrac{\ell+1}{2})\\
&=\frac{2\pi^{\sfrac{d}{2}}\Gamma(\sfrac{\ell+1}{2})}{\Gamma(\sfrac{\ell}{2})\Gamma(\sfrac{d+1}{2})}|\blambda|.\label{intintegrall}
\end{align}
Since $\cX$ and $\blambda$ depend on $(\bu_1,\dots\bu_{p-1})$, $d$ and $\ell$ potentially do as well.  One can see that $d=q-p+1$.  To analyze $\ell$, first notice that $[\blambda]\subseteq\cX$, so $\ell=\text{dim}([\blambda])$.  Also, for almost every $(\bu_1,\dots\bu_{p})\in\cU^{p}_\perp$, the subspaces $[\bu_1\wedge\dots\wedge\bu_p]$ and $[\bmu]$ are not orthogonal and, hence, $\blambda=(\uu_1\wedge...\wedge \uu_{p-1})\lrcorner \bmu$ is a nonzero $(k-p+1)$-vector.  Thus, generally, $\ell=k-p+1$.  Using these values for $d$ and $\ell$, we can substitute \eqref{intintegrall} into \eqref{outside} to obtain the desired result.
\end{proof}

\begin{lemma}\label{vfbound} Given a simple, unit $k$-vector $\bnu\in\hat\Lambda^k_u(\Real^n)$, we have for $0\leq k<n$
\beqn\label{lemma4.1eqn}
\int_{\W}|\bnu\cdot \bom|d\H^m(\textbf{a},\bom) 
=\frac{2^{\sfrac{k+4}{2}} \pi^{\frac{(n+2)(k+1)}{2}} \Gamma(\frac{n-k+1}2) }{\sqrt{\pi} \Gamma(\frac{n+1}2)}\prod_{i=1}^{k+1}\frac{\Gamma (\frac{i}2)}{\pi^i\Gamma (\frac{n-i+1}2)},
\eeqn
where $m=\text{\rm dim}\,\cW=n-1+k(n-k-1)$.
\end{lemma}

\begin{proof}First we handle the $k=0$ case.  As stated in Remark~\ref{keq0}, by definition $\hat\Lambda^0_u(\Real^n)=\{-1,1\}$ and $[\pm1]^\perp=\Real^n$.  It follows that
\beqn
\cW=\{(\ba,\bomega) : \ba\in\cU(\Real^n), \bomega=\{-1,1\}\}.
\eeqn
and, thus,
\beqn
\int_\cW |\bnu\cdot\bomega|d\cH^m(\ba,\bomega)=2\int_{\cU(\Real^n)}d\cH^{n-1}=2\omega_{n-1}.
\eeqn
This is \eqref{lemma4.1eqn} when $k=0$.

To establish the formula for $1\leq k<n$ we will use repeated applications of the coarea formula. First we convert the integral over $\W$ into an iterated integral over $\ba\in\cU(\R^{n})$ and $\bomega\in\hat{\Lambda}^k_u([\aa]^\perp)$. We then use the fact that every element of $\hat{\Lambda}^k_u([\ba]^\perp)$ is the wedge product of $k$ orthogonal vectors in $\cU([\ba]^\perp)$ to relate the integral over $\hat{\Lambda}^k_u([\ba]^\perp)$ to an integral over $\mathcal{U}^k_\perp([\ba]^\perp)$.  This integral can then be computed with repeated applications of Lemma~\ref{itlemma}.  Finally, we compute the integral over $\ba\in\cU(\R^n)$ using Lemma~\ref{subspint}.\\

\noindent\textit{Step 1}: We will convert the integral over  $\W$ into an iterated integral.  Let $F:\W\rightarrow \cU(\R^{n})$ be defined by $F(\mathbf{a},\bom)=\mathbf{a}$ for all $(\ba,\bom)\in\cW$.  Fix some $(\ba,\bom)\in\cW$ and notice that since $\bom\in \hat{\Lambda}_u^k(\R^n)$, we can write $\bom=\uu_1\wedge...\wedge \uu_k$, where $\{\uu_1,...,\uu_k\}$ is an orthonormal basis for $[\bomega]$. It follows that, because $\textbf{a}\lrcorner \bom=\mathbf{0}$, we can extend the set $\{\aa, \uu_1,...,\uu_k\}$ into an orthonormal basis for $\R^n$ by adding $n-k-1$ vectors: $\{\aa, \uu_1,...,\uu_k, \ee_1,...,\ee_{n-k-1}\}.$
Using this basis we can construct an orthonormal basis for $T_{(\aa,\bom)}\W$ which, using the notation in \eqref{replacewedgeID}, can be written as the union of the sets
\begin{enumerate}
\item $\{(\ee_i,\mathbf{0}):1\le i\le n-k-1\}$,
\item $\{(\mathbf{0},\ee_i\wedge(\uu_j\lrcorner\bom)):1\le i\le n-k-1, 1\le j\le k\}$,
\item $\{\frac{1}{\sqrt{2}}(\uu_j,-\aa\wedge(\uu_j\lrcorner\bom)):1\le j\le k\}$.
\end{enumerate}
Notice that
\begin{enumerate}
\item $\nabla F(\ba,\bomega)(\ee_i,\mathbf{0})=\ee_i$,
\item $\nabla F(\ba,\bomega)(\mathbf{0},\ee_i\wedge(\uu_j\lrcorner\bom))=\mathbf{0}$,
\item $\nabla F(\ba,\bomega)\big(\frac{1}{\sqrt{2}}(\uu_j,-\aa\wedge(\uu_j\lrcorner\bom))\big)=\frac{1}{\sqrt{2}}\uu_j$.
\end{enumerate}
Since an orthonormal basis for $T_\aa \mathcal{U}(\R^n)$ is $\{\uu_1,...,\uu_k,\be_1,..,\be_{n-k-1}\}$, we can compute the matrix $[\nabla F]$ of $\nabla F$ relative to these bases, ordered as in the lists above, as
\[
[\nabla F]= \begin{blockarray}{cccc}
n-k-1& k(n-k-1) &k&\\
\begin{block}{(ccc)c}
\bzero&\bzero&\frac{1}{\sqrt{2}}\textbf{1}_k &k\\
\textbf{1}_{n-k-1}&\bzero&\bzero&n-k-1\\
  \end{block}
\end{blockarray}.
\]
Thus,
\[
[\nabla F][\nabla F]^T=
\begin{pmatrix} 
\frac{1}2\textbf{1}_k &\bzero\\
\bzero&\textbf{1}_{n-k-1}\\
\end{pmatrix}
\quad \mbox{and so}\quad
\det(\nabla F\nabla F^T)=\frac{1}{2^k}.
\]

Recall that $\bP_\aa$ is the projection onto $[\aa]$ and $\bP_\aa^\perp$ is the projection onto $[\aa]^{\perp}$. Notice that $\bnu\cdot \bom=\bP_\aa^\perp \bnu\cdot \bom$, since $\aa\lrcorner \bom=\textbf0.$
Using $F$ with the coarea formula, we have 
\begin{equation}\label{firstbreak}
\int_{\W}|\bnu\cdot \bom|d\H^m(\textbf{a},\bom) =2^{\frac{k}{2}}\int_{\cU(\R^n)}\int_{\hat{\Lambda}^k_u([\ba]^\perp)}|\bP_{\aa}^\perp\bnu\cdot\bom|d \H^{k(n-k-1)}(\bom)d\H^{n-1}(\aa).
\end{equation}

\noindent\textit{Step 2}:
We utilize the fact that every element of $\hat{\Lambda}^k_u([\aa]^\perp)$ is the wedge product of $k$ orthogonal vectors in $\cU([\aa]^\perp)$ to relate the integral over $\hat{\Lambda}^k_u([\aa]^\perp)$ to an integral over $\mathcal{U}^k_\perp([\aa]^\perp)$.
 Define $G: \mathcal{U}^k_\perp([\aa]^\perp)\rightarrow \hat{\Lambda}_u^k([\aa]^\perp)$ by 
\[
G(\uu_1,...,\uu_k)=\uu_1\wedge...\wedge \uu_k,\qquad (\bu_1,\dots\bu_k)\in\cU^k_\perp([\ba]^\perp).
\] 
For $\bom\in \hat{\Lambda}^k_u([\ba]^\perp)$, fix $(\bu_1,\dots,\bu_k)\in\cU_\perp^k([\aa]^\perp)$ such that $\bomega=\bu_1\wedge\dots\wedge\bu_k$. Extend the set $\{\uu_1,...,\uu_k\}$ to an orthonormal basis for $[\aa]^\perp$ as $\{\uu_1,...,\uu_k,\ee_1,...,\ee_{n-k-1}\}.$ Then an orthonormal basis for $T_{(\uu_1,...,\uu_k)}\mathcal{U}^k_{\perp}([\aa]^\perp)$ is given just as in Lemma~\ref{itlemma} with $p=k$ and $q=n-1$.
One can calculate
\beqn
\nabla G(\uu_1,...,\uu_k) \bE_{ij}=\ee_j\wedge(\uu_i\lrcorner\bom)\quad \text{and}\quad \nabla G(\uu_1,...,\uu_k) \bF_{ij}=\textbf{0}.
\eeqn
Indeed, the latter term is zero because adjusting $(\uu_1,...,\uu_k)$ in the $\bF_{ij}$ direction is a rotation, so the wedge product of the adjusted vector will remain unchanged. 

Moreover, the set $\{\ee_j\wedge(\uu_i\lrcorner\bom):1\le i\le k, 1\le j\le n-k-1\}$ consists of $k(n-k-1)$ many elements of $T_{\bom} \hat{\Lambda}_u^k([\aa]^\perp)$, so it forms an orthonormal basis for the set.  The matrix $[\nabla G]$ of $\nabla G$ relative to these bases is
\[
[\nabla G]= 
\begin{blockarray}{cc}
k(n-k-1)& k(k-1)/2\\
\begin{block}{(cc)}
\textbf{1}_{k(n-k-1)} & \textbf{0}\\
\end{block}
\end{blockarray}
\qquad \text{and so}\qquad \det(\nabla G\nabla G^T)=1.
\]

The preimage $G^{-1}(\bom)$ is the set of elements in $\mathcal{U}_\perp^k([\ba]^\perp)$ related to our fixed $(\bu_1,\dots,\bu_k)$ by a rotation. So $\H^{\sfrac{k(k-1)}{2}}(G^{-1}(\bom))=\H^{\sfrac{k(k-1)}{2}}(\SO(k))$.  Thus, by the coarea formula, the inner integral on the right-hand side of \eqref{firstbreak} becomes
\begin{multline}
\int_{\hat{\Lambda}^k_u([\ba]^\perp)}|\bP_{\aa}^\perp\bnu\cdot\bom|\, d \H^{k(n-k-1)}(\bom)\\
\label{secondbreak}=\frac{1}{\H^{\sfrac{k(k-1)}{2}}(\SO(k))}\int_{\mathcal{U}_\perp^k([\ba]^\perp)} |\bP_{\aa}^\perp\bnu\cdot \uu_1\wedge...\wedge \uu_k| \, d\H^{\sfrac{k(2n-k-3)}{2}}(\uu_1,...,\uu_k).
\end{multline}

\noindent\textit{Step 3}: We simplify the right-hand side of \eqref{secondbreak} by calculating
 \[
\int_{\mathcal{U}_\perp^k([\ba]^\perp)} |\bP_{\aa}^\perp\bnu\cdot \uu_1\wedge...\wedge \uu_k| \, d\H^{\sfrac{k(2n-k-3)}{2}}(\uu_1,...,\uu_k)
\]
using Lemma~\ref{itlemma}.  To do so, first notice that
\beqn
\bP_{\aa}^\perp\bnu\cdot \uu_1\wedge...\wedge \uu_k=(\bu_1\wedge\dots\wedge\bu_k)\lrcorner\bP^\perp_\ba\bnu.
\eeqn
Thus, we can use Lemma~\ref{itlemma} with $\cZ=[\ba]^\perp$, $\bmu=\bP^\perp_\ba\bnu$, and $p=k$ to find that
\begin{multline}
\int_{\mathcal{U}_\perp^k([\ba]^\perp)} |\bP_{\aa}^\perp\bnu\cdot \uu_1\wedge...\wedge \uu_k| \, d\H^{\sfrac{k(2n-k-3)}{2}}(\uu_1,...,\uu_k)\\
=C(n,k,k) \int_{\mathcal{U}_\perp^{k-1}([\ba]^\perp)} |(\bu_1\wedge\dots\wedge\bu_{k-1})\lrcorner\bP^\perp_\ba\bnu| \, d\H^{\sfrac{(k-1)(2n-k-2)}{2}}(\uu_1,...,\uu_{k-1}),
\end{multline}
where
\beqn
C(n,k,i)=\frac{2^{\sfrac{i+1}{2}}\pi^{\sfrac{n-i}{2}}\Gamma(\sfrac{k-i+2}{2})}{\Gamma(\sfrac{k-i+1}{2})\Gamma(\sfrac{n-i+1}{2})}.
\eeqn
We can keep applying Lemma~\ref{itlemma} to obtain
\beqn\label{almostsecond}
\int_{\mathcal{U}_\perp^k([\ba]^\perp)} |\bP_{\aa}^\perp\bnu\cdot \uu_1\wedge...\wedge \uu_k| \, d\H^{\sfrac{k(2n-k-3)}{2}}(\uu_1,...,\uu_k)=\prod_{i=1}^kC(n,k,i) |\bP_{\aa}^\perp\bnu|.
\eeqn
A calculation using \eqref{gammahalf} shows that
\beqn\label{defK}
K\coloneqq\prod_{i=1}^kC(n,k,i)=\frac{2^{\sfrac{k(k+3)}4}\pi^{\sfrac{k(2n-k-1)-2}4}\Gamma(\sfrac{k+1}2)}{\prod_{i=1}^k\Gamma(\frac{n-i+1}2)}.
\eeqn
Putting together \eqref{firstbreak}, \eqref{secondbreak}, and \eqref{almostsecond}, we see that 
\begin{equation}\label{thirdbreak}
\int_{\W}|\bnu\cdot \bom|d\H^m(\textbf{a},\bom) =\frac{2^{\frac{k}{2}}K}{\cH^{\sfrac{k(k-1)}{2}}(\SO(k))}\int_{\cU(\R^n)}  |\bP_\aa^\perp\bnu| d\H^{n-1}(\aa).
\end{equation}

\noindent\textit{Step Four}: We complete the integral calculation by decomposing $\cU(\R^n)$ into $\cU([\bnu])$ and $\cU([\bnu]^\perp)$ via Lemma \ref{subspint}. By \eqref{kvectorid}$_2$ and Lemma \ref{subspint}, we have:
\begin{align}
\int_{\cU(\R^n)}  |\bP_\aa^\perp\bnu| d\H^{n-1}(\aa)&=\int_{\cU([\bnu])}\int_{\cU([\bnu]^{\perp})}\int_0^{\pi/2}\sin\theta \cos^{k-1}\theta\sin^{n-k-1}\theta \, d\theta d\H^{n-k-1}d\H^{k-1}\nonumber\\
&=\omega_{k-1}\omega_{n-k-1}\sfrac{1}{2}B\left(\sfrac{n-k+1}{2},\sfrac{k}{2}\right)\nonumber\\
&=\frac{2\pi^{\frac{n}{2}}\Gamma(\frac{n-k+1}2)}{\Gamma(\frac{n+1}{2})\Gamma(\frac{n-k}{2})}.\label{lastsplit}
\end{align}
Inserting \eqref{lastsplit} and \eqref{defK} into \eqref{thirdbreak}, we see that
\beqn
\int_{\W}|\bnu\cdot \bom|d\H^m(\textbf{a},\bom) = \frac{2^{\sfrac{(k+4)(k+1)}{4}}\pi^{\sfrac{2kn-k^2+2n-k-2}{4}}\Gamma(\sfrac{k+1}{2})\Gamma(\frac{n-k+1}2)}{\Gamma(\sfrac{n+1}{2})\Gamma(\sfrac{n-k}{2})\cH^{\sfrac{k(k-1)}{2}}(\SO(k))\prod_{i=1}^k\Gamma(\frac{n-i+1}2)}.
\eeqn
Now substiuting in \eqref{SOnmeas} yields the desired result. 
\end{proof}

\begin{theorem}\label{thmlimsmeas}
If $\Om$ is an open, bounded set such that $\M\subseteq \Om$, then 
\begin{equation}
\lim_{\s\uparrow 1} (1-\s)\text{\rm Meas}_{\s}^k(\M,\Om)=\frac{4 \pi^{\frac{(n+2)(k+1)}{2}} \Gamma(\frac{n-k+1}2) }{\sqrt{\pi} \Gamma(\frac{n+1}2)(n-k)}\prod_{i=1}^{k+1}\frac{\Gamma (\frac{i}2)}{\pi^i\Gamma (\frac{n-i+1}2)}\H^k(\M).
\end{equation}

\end{theorem}
\begin{proof} Set $\e\coloneqq(1-\s)^{1/n}$ and 
\[
\D_\e(\M)\coloneqq\{(p,\bom, r)\in \D(\M):r\le \e\}.
\]
Notice that, for sufficiently small $\e$, 
\begin{equation}
\D_\Om(\M)\setminus \D_\e(\M)\subseteq \Xi\Big( \bigcup_{\xi\in\R^+} \M\times \W\times \{\xi\}\times [\max\{\xi,\e\},\xi+\text{diam}(\Om)]\Big)
\end{equation}
by an argument similar to the one justifying \eqref{oddballs}. Thus, using the change of variables formula from Lemma \ref{COV}, we know there is a constant $C_n$ depending on $\M$ and $n$ such that 

\begin{align}
\nonumber\int_{\D_\Om(\M)\setminus\D_\e(\M)}r^{k-n-\s}&d\H^{2n}(p,\bom,r)\le C_n \int_0^\infty\int_{\max\{\xi,\e\}}^{\xi+\text{diam}(\Om)}  r^{k-n-\s}\xi^{n-k-1}dr d\xi\\
\nonumber=&\frac{C_n}{n+\s-k-1}\left[\int_0^\e \xi^{n-k-1}(\e^{k-n-\s+1}-(\xi+\text{diam}(\Om))^{k-n-\s+1} )d\xi\right.\\
\label{Donote}&\qquad+\left.  \int_\e^\infty \xi^{-\s}-\xi^{n-k-1}(\xi+\text{diam}(\Om))^{k-n-\s+1}d\xi  \right].
\end{align} 
Looking at the first integral on the right-hand side of \eqref{Donote} we find that
\[
\int_0^\e \xi^{n-k-1}(\e^{k-n-\s+1}-(\xi+\text{diam}(\Om))^{k-n-\s+1} )d\xi \le \frac{\e^{1-\s}}{n-k}\,,
\]
which goes to $1/(n-k)$ as $\s\uparrow 1.$ From \eqref{xibound}, we also see that the second integral in \eqref{Donote}, when multiplied by $1-\sigma$, goes to zero as $\s\uparrow 1.$ Hence 
\begin{equation}\label{lim0}
\lim_{\s\uparrow 1}(1-\s)\int_{\D_\Om(\M)\setminus\D_\e(\M)}r^{k-n-\s}d\H^{n+k(n-k)+1}(p,\bom,r)=0.
\end{equation}

For each $(p,\bom,r)\in\D_{\e}(\M)$ we know there is at least one $z\in\M$ so that $D(p,\bom,r)$ intersects $\M$ at $z$. Define $f(p,\bom,r)\in\M$ to be one such chosen point for each 
$(p,\bom,r)\in\D_{\e}(\M)$. For $z\in\M$ and $(\aa,\bom)\in\W$ set 
\[
R_\e(z,\aa,\bom)\coloneqq\{(\xi,r)\in \R^+\times\R^+:(z+\xi\aa,\bom,r)\in \cD_{\e}(\M)\text{ and }f(z+\xi\aa,\bom,r)=z\}.
\]
We have defined $f$ so that $\Xi$ is injective on
\[
\bigcup_{(z,\aa,\bom)\in \M\times \W}\{z\}\times\{\aa\}\times\{\bom\}\times R_\e(z,\aa,\bom).
\] 
To see why this is the case, consider two points $(z_1,\ba_1,\bomega_1,\xi_1,r_1)$ and $(z_2,\ba_2,\bomega_2,\xi_2,r_2)$ in this set such that
\beqn
\Xi(z_1,\ba_1,\bomega_1,\xi_1,r_1)=\Xi(z_2,\ba_2,\bomega_2,\xi_2,r_2).
\eeqn
It follows that
\beqn\label{samedisk}
(z_1+\xi_1\ba_1,\bomega_1,r_1)=(z_2+\xi_2\ba_2,\bomega_2,r_2).
\eeqn
From here we can conclude that $\bomega_1=\bomega_2$ and $r_1=r_2$.  It also follows that
\beqn
f(z_1+\xi_1\ba_1,\bomega_1,r_1)=f(z_2+\xi_2\ba_2,\bomega_1,r_2)=:z.
\eeqn
From the definition of $f$, this implies that $z=z_1=z_2$.  Thus, from \eqref{samedisk} we obtain $\xi_1\ba_1=\xi_2\ba_2$.  Since $\xi_1$ and $\xi_2$ are positive and $\ba_1$ and $\ba_2$ are unit vectors, this is only possible if $\xi_1=\xi_2$ and $\ba_1=\ba_2$.

Again applying the change of variables formula in Lemma~\ref{COV}, we get 
\begin{multline}\label{lim1}
\int_{\D_\e(\M)}r^{k-n-\s}\,d\H^{n+k(n-k)+1}(p,\bom,r)\\
=\int_\M\int_{\W}\int_{R_\e(z,\aa,\bom)} r^{k-n-\s} \frac{\xi^{n-k-1}}{2^{k/2}}|\vol_{\M}(z)\cdot \bomega|\, d\H^2(\xi,r)d\H^{n-1+k(n-k-1)}(\aa,\bom)dz.
\end{multline}
Since $\M$ is a smooth manifold, if $(z,\aa,\bom)$ is such that $\ba\not\in [\textbf{vol}_\cM(z)]$ and $\bomega\cdot\textbf{vol}_\cM(z)\not=\bzero$, then for any $\xi,r>0$ the disk $D(z+\xi\ba,\bomega,r)$ is not tangent to $\cM$ at $z$.  It follows that there is $\e_0$ such that, for $\e\le \e_0$ any disk $D(z+\xi\ba,\bomega,r)$ for $\xi\in[0,\e]$ and $r\in[\xi,\e]$ only intersects $\cM$ at $z$.  Thus, for such $\e$ we have
\[
R_\e(z,\aa,\bom)=\{(\xi,r)\in\R^+\times\R^+:\xi\in[0,\e], r\in[\xi,\e]\}.
\]
Then, integrating and using the fact that $\displaystyle\lim_{\s\uparrow 1} \e^{1-\s}=1$, we have
\begin{align*}
\lim_{\s\uparrow 1} (1-\s)\int_{R_\e(z,\aa,\bom)} r^{k-n-\s} \xi^{n-k-1}d\H^2(\xi,r) &=  \lim_{\s\uparrow 1} (1-\s)\int_{0}^\e\int_\xi^\e r^{k-n-\s} \xi^{n-k-1}dr d\xi\\
&=\frac{1}{n-k}.
\end{align*}

From Item~\ref{tan} in Lemma~\ref{lemmeasure}, we know that the measure of the set of disks that are tangent to $\cM$ is negligible.  Thus, we can combining the calculation in the previous equation with \eqref{lim0} and \eqref{lim1} we obtain
\[
\lim_{\s\uparrow 1} \text{Meas}_\s^k(\M,\Om)=\frac{2^{-k/2}}{n-k}  \int_\M\int_{\W} |\vol_{\M}\cdot \bomega|\,d\H^{n-1+k(n-k-1)}(\aa,\bom)dz.
\]
Combining this with Lemma \ref{vfbound}, we have our result. 
\end{proof}

We use the term ``fractional measure'' for $\text{Meas}_\sigma^k$ since it generalizes the concepts of fractional perimeter, fractional area, and fractional length appearing in the literature.  However, it is important to keep in mind that the $\sigma$-measure relative to a bounded set $\Omega$ is not a measure because it is not additive.  To see this, we will explicitly compute the $\sigma$-measure for a few sets in the $n=1$, $k=0$ case.  Setting $\Omega=(-2,2)$, we will show that
\beqn\label{nonmeas}
\text{Meas}^0_\sigma(\{-1,1\},\Omega)< \text{Meas}^0_\sigma(\{-1\},\Omega)+\text{Meas}^0_\sigma(\{1\},\Omega).
\eeqn
In the $n=1$, $k=0$ case, the disks are oriented open intervals.  In order for a disk to intersect $\{-1,1\}$ an odd number of times, it must contain either $-1$ or $1$, but not both.  If the center of an interval $p$ satisfies $p<0$, then the interval contains only $-1$ if its radius $r$ satisfies $|p+1|<r\leq |1-p|$, while if $p>0$, then the interval contains only $1$ if $|1-p|<r\leq |p+1|$.  Moreover, any such interval satisfying these conditions has an endpoint inside $\Omega$.  Thus, we have
\beqn\label{calc1}
\text{Meas}^0_\sigma(\{-1,1\},\Omega)=2\Big(\int_{-\infty}^0\int_{|p+1|}^{|1-p|}+\int^{\infty}_0\int^{|p+1|}_{|1-p|}\Big)r^{-1-\sigma}drdp=\frac{8}{\sigma(1-\sigma)}.
\eeqn
The factor of two in front of the integral comes from the fact that we have to do the calculation for intervals of both orientations.  Next, we compute $\text{Meas}^0_\sigma(\{1\},\Omega)$.  If $p>0$, an interval centered at $p$ of radius $r$ will contain $1$ and have a boundary point in $\Omega$ when $|1-p|<r<|p+2|$, while if $p<0$ the radius of such an interval must satisfy $|1-p|<r<|2-p|$.  Thus,
\beqn\label{calc2}
\text{Meas}^0_\sigma(\{1\},\Omega)=2\Big(\int_{-\infty}^0\int_{|1-p|}^{|2-p|}+\int^{\infty}_0\int_{|p-1|}^{|p+2|}\Big)r^{-1-\sigma}drdp=\frac{2^{3-\sigma}}{\sigma(1-\sigma)}.
\eeqn
From symmetry, it follows that
\beqn\label{calc3}
\text{Meas}^0_\sigma(\{-1\},\Omega)=\text{Meas}^0_\sigma(\{1\},\Omega).
\eeqn
Putting together \eqref{calc1}--\eqref{calc3} yields \eqref{nonmeas}.

\section{Future directions}\label{Sect5}

The fractional notion of measure introduced here opens the door to numerous avenues of research.  Below we outline a few of them.

One of the classical problems in geometric measure theory is Plateau's problem.  In fact, it drove the innovation in this field for many decades.  See Harrison and Pugh \cite{HP16} for an overview of the history of this problem.  In its most classical form, motivated by the study of soap films, the question is: given a space curve in three dimensions, is there a spanning surface\footnote{In this section we use terms like ``surface'' and ``manifold'' loosely.  This is because the objects with minimal area need not be manifolds in the strict sense, but can have singularities such as triple junctions.} with minimal area.  Generalized to $n$ dimensions, the problem can be phrased as given an $(n-2)$-dimensional manifold $\cN$ with empty boundary, is there an $(n-1)$-dimensional surface $\cM$ such that $\partial\cM=\cN$ with minimal $\cH^{n-1}$-measure.  The fractional version of this problem was studied by Caffarelli, Roquejoffre, and Savin \cite{CRS10} and has spawned a plethora of works on nonlocal minimal surfaces, as mentioned in the introduction.  One can generalize Plateau's problem even further by considering $\cN$ to have dimension $k-1$ and looking for a spanning manifold of dimension $k$ with minimal $\cH^{k}$-measure.  The corresponding fractional version would be phrased as follows.  Given a $k$-dimensional manifold $\cN$ with empty boundary and an open set $\Omega$ containing $\cN$, show there exists a minimizer of the functional
\beqn\label{PP}
J_\sigma(\cM)=\text{Meas}^k_\sigma(\cM,\Omega)
\eeqn
over all $\cM$ such that $\partial\cM=\cN$.  Besides establishing existence, the regularity of the minimizers could also be investigated.

Another classical problem is the isoparametric inequality.  In its most basic form it states that if a plane curve of length $L$ encloses a region of area $A$, then
\beqn
4\pi A\leq L^2,
\eeqn
with equality holding only in the case when the curve is a circle.  In $n$ dimensions the corresponding inequality involving sets $E$ of finite $\cH^n$-measure is
\beqn
n^n\alpha_n\cH^n(E)^{n-1}\leq \cH^{n-1}(\partial E)^n,
\eeqn
with equality holding if and only if $E$ is a ball.  A version of this result involving the fractional perimeter was established by Frank and Seiringer \cite{FS08}.  Similar to Plateau's problem, the isoparametric inequality one can be generalized to arbitrary codimension.  As established by Almgren \cite{A86}, for each $k\leq n$, there is a $k$-dimensional manifold $\cM\subseteq\Real^n$ with boundary such that
\beqn
k^k\alpha_k\cH^k(\cM)^{k-1}\leq \cH^{k-1}(\partial\cM)^k
\eeqn
and equality holds if and only if $\cM$ is a $k$-dimensional disk.  Whether this results holds, perhaps in a modified form, when using the fractional notion of measure introduced here could be investigated.  

Much of the analysis of Plateau's problem and the isoparametric inequality involves using objects more general than a manifold, such as flat chains, integral currents, and varifolds.  Thus, the study of the fractional version of these statements will probably require extending the definition of $\text{Meas}_\sigma^k$ to such objects.  This is a topic worth investigating in its own right.

Finally, the fractional measure should be able to motivate a nonlocal mean-curvature vector.  As mentioned in the introduction, $\sigma$-minimal surfaces must satisfy a pointwise condition on their boundary which Abatangelo and Valdinoci \cite{AV14} used to define nonlocal notions of mean and directional curvature for surfaces.  Moreover, the nonlocal notion of directional curvature motivates a nonlocal curvature tensor, or second fundamental form, for surfaces, as shown by Paroni, Podio-Guidugli, and Seguin \cite{PPGS23}.  Upon introducing a fractional notion of length, Seguin \cite{S20,S20c} found that curves that minimize this length while keeping the boundary points fixed must satisfy a pointwise condition which he used to define a nonlocal notion of curvature for a curve.  In a similar way, it is likely the case that any minimizer of \eqref{PP} would satisfy a pointwise condition that could be used to define a nonlocal mean-curvature vector for manifolds.

\appendix

\section{Appendix}

Let $\mathcal{M}$ be a bounded $k$-dimensional manifold in $\mathbb{R}^n$ whose boundary is a $(k-1)$-dimensional manifold.  Recall the definition of $\cW$ in \eqref{Wdef} and that $\textbf{vol}_\cM$ is a volume form for $\overline\cM$.  To describe the disks that intersect $\overline{\M}$ we consider the function
\[
\Xi:\overline{\M}\times\W\times \R^+_0\times \R^+\rightarrow \R^n \times \hat{\Lambda}^k_u(\R^n)\times\R
\]
defined by
\begin{equation}\label{xi}
\Xi(z,\textbf{a}, \bom, \xi, r)\coloneqq(z+\xi \textbf{a},\bom, r),\qquad (z,\textbf{a},\bom,\xi,r)\in \overline\M\times\W\times \R^+_0\times\R^+.
\end{equation}
Notice that $\Xi(z,\textbf{a}, \bom, \xi, r)$ is the disk with center $z+\xi\mathbf{a}$, those normal $k$-vector $\bomega$ is orthogonal to $\ba$, and has radius $r$.

\begin{lemma}\label{COV}
If $\A$ is a subset of $\M\times\W\times \R^+\times \R^+$ and $f:\Xi(\A)\rightarrow \R$ is an integrable function, then 
\begin{multline}\label{COVlemma}
\int_{\Xi(\A)}\Big[\sum_{(z,\mathbf{a},\bom,\xi,r)\in\Xi^{-1}(p,\bom,r)}f(p,\bom, r)\Big] \, d\H^{n+k(n-k)+1}(p,\bom, r)\\ 
=\int_{\A} f(z+\xi \mathbf{a}, \bom, r) \frac{\xi^{n-k-1}}{2^{k/2}}|\vol_{\M}(z)\cdot \bomega|\,d\H^{n+k(n-k)+1}(z,\mathbf{a}, \bom, \xi, r).
\end{multline}
 
\end{lemma}

\begin{proof}
It suffices to prove the result for a set of the form $\A=\M_\A\times \mathcal{W}_\A\times \mathcal{I}_A$, where $\M_\A\subseteq \M$, $\mathcal{W}_A\subseteq \W,$ and $\mathcal{I}\subseteq\R^+$.  Moreover, by employing a partition of unity, we can reduce the result to the case where $\M_\A$ and $\mathcal{W}_\A$ are each covered by one chart. That is, there are $M_A\subseteq \R^k$ and $W_A\subseteq\R^m$, where $m\coloneqq\text{dim}\,\cW=n-1+k(n-k-1)$, and diffeomorphisms $\phi:M_A\rightarrow \M_\A\subseteq\Real^n$ and $\bchi:W_A\rightarrow \mathcal{W}_\A\subseteq\Real^{n+{n\choose k}}$. Since $ \cW_\A\subseteq \cW$, we can view $\bchi$ in terms of it component functions: $\bchi=(\bchi_1,\bchi_2),$ where $\bchi_1:\R^m\rightarrow \cU(\R^n)\subseteq\R^n$ and $\bchi_2:\R^m\rightarrow\hat{\Lambda}^k_u(\R^n)\subseteq\R^{n\choose k}$. 

Recall that, if $g$ is an integrable function defined on $\mathcal{W}_\A$, then 
\begin{equation}\label{perpchange}
\int_{\mathcal{W}_\A}g(\mathbf{a},\bom)d \H^{m}(\mathbf{a},\bom)=\int_{W_A}g(\bchi_1(w), \bchi_2(w))J_{\bchi}(w)dw,
\end{equation}
where $J_{\bchi}=\sqrt{\det(\nabla\bchi^T\nabla \bchi)}$ is the Jacobian of $\bchi.$ 
Also, if $h$ is an integrable function defined on $\M_\A$, then 
\begin{equation}\label{surfchange}
\int_{\M_\A}h(z)d\H^{k}(z)=\int_{M_A}h(\phi(x)) J_\phi(x)dx.
\end{equation}

Set $A\coloneqq M_A\times W_A\times \mathcal{I}_A\times \Real^+$, and define $F:A\rightarrow\Xi(\mathcal{A})$ by 
\[
F(x,w,\xi,r)\coloneqq(\phi(x)+\xi\bchi_1(w),\bchi_2(w),r),\qquad (x,w,\xi,r)\in A.
\]
To establish \eqref{COVlemma} by the area formula, we need to compute $J_{F}=\sqrt{|\det( \nabla F^T\nabla F) | }$.\\

\noindent\textit{Step 1}: Simplify $J_F$ using block matrices.  To do this, first notice that we have
\[
\nabla F=
\begin{blockarray}{ccccc}
k & m & 1 & 1  \\
\begin{block}{(cccc)c}
  \nabla \phi & \xi\nabla \bchi_1 & \bchi_1 & \bzero  & n \\
  \bzero & \nabla\bchi_2 & \bzero & \bzero  &{n\choose k} \\
  \bzero & \bzero & 0 & 1 & 1 \\
\end{block}
\end{blockarray}
\qquad\qquad 
\nabla F^T=
\begin{blockarray}{cccc}
n & {n\choose k}  & 1  \\
\begin{block}{(ccc)c}
  \nabla \phi^T &\bzero & \bzero  & k \\
   \xi\nabla \bchi_1^T & \nabla\bchi_2^T & \bzero  & m\\
  \bchi_1^T & \bzero & 0 & 1 \\
    \bzero & \bzero & 1  & 1 \\
\end{block}
\end{blockarray}
\]
and, because $\nabla \bchi_1^T \bchi_1=\bzero$ and $\bchi_1^T\bchi_1=1$,
\[
\nabla F^T\nabla F=
\begin{blockarray}{ccccc}
k & m & 1 & 1  \\
\begin{block}{(cccc)c}
  \nabla \phi ^T \nabla \phi  & \xi\nabla \phi^T\nabla \bchi_1 & \nabla \phi^T\bchi_1 & \bzero  & k \\
  \xi\nabla \bchi_1^T\nabla \phi & \xi^2\nabla \bchi_1^T\nabla\bchi_1+\nabla \bchi_2^T\nabla\bchi_2 & \bzero & \bzero  & m \\
  \bchi_1^T\nabla \phi & \bzero & 1 & 0 & 1 \\
  \bzero & \bzero & 0 & 1  & 1 \\
\end{block}
\end{blockarray}.
\]
To compute the absolute value of the determinant of this matrix, we rearrange the rows and columns into a block matrix and apply Schur's formula:
\renewcommand{\arraystretch}{1.15}
\begin{align}
|\det(\nabla F^T\nabla F)|&=\left|\det\,
\begin{blockarray}{cccc}
k & 1 & m  \\
\begin{block}{(cc|c)c}
  \nabla \phi ^T \nabla \phi & \nabla \phi^T\bchi_1 & \xi\nabla \phi^T\nabla \bchi_1   & k \\
  \bchi_1^T\nabla \phi  & 1& \bzero & 1 \\
  \cline{1-3}
 \xi\nabla \bchi_1^T\nabla \phi & \bzero& \xi^2\nabla \bchi_1^T\nabla\bchi_1+\nabla \bchi_2^T\nabla\bchi_2   & m \\ 
\end{block}
\end{blockarray}\right|\\
&=\Biggl|\det\,
\begin{blockarray}{cc}
&\\
\begin{block}{(c|c)}
\bA &\bB \\ 
\cline{1-2}
 \bB^T& \bC \\
  \end{block}
\end{blockarray}\Biggl|\\
\label{JFcomp}&= |\det(\bA)\det(\bC-\bB^T\bA^{-1}\bB)|.
\end{align}
\renewcommand{\arraystretch}{1}

\noindent\textit{Step 2}: Compute $\det \bA$.  To accomplish this, we introduce some notation.  Recall that $\nabla\phi(x)\in\Lin(\Real^k,\Real^n)$, but its range is $T_{\phi(x)}\cM$.  Let $\bP_{\phi(x)}$ denote the projection of $\Real^n$ onto $T_{\phi(x)}\cM$ so that $\bPhi(x)=\bP_{\phi(x)}\nabla\phi(x)$ is a linear isomorphism from $\Real^k$ to $T_{\phi(x)}\cM$.  Since we often suppress the argument $x$, we will write $\bPhi\in\Lin(\Real^k,T\cM)$, where it is implicitly understood at which point the tangent space is being considered.  Continuing this abuse of notation, we denote by $\bone_{T\cM}\in\Lin(T\cM,T\cM)$ the identity function on the tangent space $T_{\phi(x)}\cM$ for the appropriate $x$ without explicitly writing it.  Using this notation and properties of block matrix determinants we have
\begin{align}\det \bA=&\det
\begin{pmatrix}
\nabla \phi ^T \nabla \phi & \nabla \phi^T\bchi_1 \\
\bchi_1^T\nabla \phi  & 1 \\
  \end{pmatrix}
=\det\begin{pmatrix}
 \bPhi ^T \bPhi & \bPhi^T\bP_\phi\bchi_1 \\
   (\bPhi^T\bP_\phi\bchi_1)^T & 1 \\
\end{pmatrix}\nonumber\\
=& \det(\bPhi^T\bPhi-\bPhi^T\bP_\phi\bchi_1(\bP_\phi\bchi_1)^T\bPhi)\nonumber\\
=& \det(\bPhi^T[\bone_{T\M}-\bP_\phi\bchi_1(\bP_\phi\bchi_1)^T]\bPhi)\nonumber\\
=& \det(\bPhi^T\bPhi)\det(\bone_{T\M}-\bP_\phi\bchi_1(\bP_\phi\bchi_1)^T). \label{Ared}
\end{align}
To complete this calculation we need the determinant of $\bN$ defined by
\begin{equation}\label{defN}
\bN\coloneqq\bone_{T\M}-\bP_\phi\bchi_1(\bP_\phi\bchi_1)^T=\bone_{T\M}-\bP_\phi\bchi_1\otimes \bP_\phi\bchi_1.
\end{equation}
To find this determinant, we construct an orthonormal basis of $\R^n$ as follows. Start by considering an orthonormal basis $\{\tt_1,...,\tt_k\}$ of $T\M$ (at the appropriate point), and then extend it to an orthonormal basis of the whole space with vectors $\{\ee_{k+1},...,\ee_n\}$. Without loss of generality, we will assume $\bchi_1$ is in the span of $\tt_k$ and $\ee_{k+1}$. The matrix representation $[\bN]$ of $\bN$ relative to this basis is:
\beqn\label{repN}
[\bN]=
\begin{blockarray}{ccccc}
\bt_1 &\bt_2&...  & \bt_{k-1}&\bt_k  \\
\begin{block}{(ccccc)}
1&0&...&0&0\\
0&1&...&0&0\\
\vdots&\vdots&\ddots&\vdots&\vdots\\
0&0&...&1&0\\
0&0&...&0&1-(\tt_k\cdot\bchi_1)^2\\
\end{block}
\end{blockarray}.
\eeqn
Setting $J\coloneqq1-(\tt_k\cdot\bchi_1)^2=1-|\bP_\phi\bchi_1|^2=\det(\bN)$, we have  
\beqn\label{detAcomp}
\det \bA= \det(\bPhi^T\bPhi)J.  
\eeqn

\noindent\textit{Step 3}: Calculation of $\det(\bC-\bB^T\bA^{-1}\bB)$.  As part of this calculation, we need $\bA^{-1}$, which involves having a formula for $\bN^{-1}$.  One can see from \eqref{repN} that  
\[
\bN^{-1}=\tt_1\otimes \tt_1+...+\tt_{k-1}\otimes \tt_{k-1}+\frac{1}{J}\tt_{k}\otimes \tt_{k}.
\]
But $\frac{1}{J}=\frac{1}{J}(J+(\tt_k\cdot\bchi_1)^2)=1+\frac{(\tt_k\cdot\bchi_1)^2}{J}$, so $\bN^{-1}=\bone_{T\M}+\frac{1}{J}\bP_\phi\bchi_1\otimes \bP_\phi\bchi_1.$ 
Therefore, using the formula for the inverse of a block matrix in Theorem 2.1(ii) of \cite{LS02},
\begin{align*}
\bA^{-1}&= \begin{pmatrix}
(\bPhi^T\bN\bPhi)^{-1}& -(\bPhi^T\bN\bPhi)^{-1}\bPhi^T\bP_\phi\bchi_1 \\
-(\bP_\phi\bchi_1)^T\bPhi(\bPhi^T\bN\bPhi)^{-1} & 1+(\bP_\phi\bchi_1)^T\bPhi (\bPhi^T\bN\bPhi)^{-1}\bPhi^T\bP_\phi\bchi_1\\
  \end{pmatrix}\\
&= \begin{pmatrix}
\bPhi^{-1}\bN^{-1}\bPhi^{-T}& -\bPhi^{-1}\bN^{-1}\bP_\phi\bchi_1 \\
-(\bP_\phi\bchi_1)^T\bN^{-1}\bPhi^{-T} & 1+(\bP_\phi\bchi_1)^T\bN^{-1} \bP_\phi\bchi_1\\
\end{pmatrix}\\
&=\frac{1}{J} \begin{pmatrix}
J\bPhi^{-1}\bN^{-1}\bPhi^{-T}& -\bPhi^{-1}\bP_\phi\bchi_1 \\
-(\bP_\phi\bchi_1)^T\bPhi^{-T} & J+|\bP_\phi\bchi_1|^2\\
\end{pmatrix}\\
&=\frac{1}{J} \begin{pmatrix}
J\bPhi^{-1}\bN^{-1}\bPhi^{-T}& -\bPhi^{-1}\bP_\phi\bchi_1 \\
-(\bP_\phi\bchi_1)^T\bPhi^{-T} & 1\\
\end{pmatrix}.
\end{align*}

Using the same abuse of notation as before, let $\bP_{\bchi}\in \Lin( \R^{n+{n\choose k}}, T \W)$ be the projection onto the range of $\nabla \bchi$ and $\bI_{\bchi} =\bP_{\bchi}^T\in \Lin\big(T \W,\R^{n+{n\choose k}}\big)$ the injection from $T\W$ into $\R^{n+{n\choose k}}$.  Using properties of block matrix determinants, we find that 
\begin{align}
\det(\bC-\bB^T\bA^{-1}\bB)&=\det(\xi^2\nabla \bchi_1^T\nabla\bchi_1+\nabla \bchi_2^T\nabla\bchi_2 - \xi^2\nabla\bchi_1^T \bP_\phi^T\bN^{-1}\bP_\phi\nabla \bchi_1 )\\
&=\det\left[\nabla\bchi^T \begin{pmatrix}
\xi^2(\textbf{1}_n-\bP_\phi^T\bN^{-1}\bP_\phi)& \textbf{0} \\
\textbf{0} & \textbf{1}_{{n\choose k}}\\
\end{pmatrix}\nabla\bchi
\right]\\
&=\det(\nabla \bchi^T\nabla \bchi)\det\left[\bI_{\bchi}^T \begin{pmatrix}
\xi^2(\textbf{1}_n-\bP_\phi^T\bN^{-1}\bP_\phi)&\textbf{0} \\
\textbf{0} & \textbf{1}_{{n\choose k}}\\
\end{pmatrix}\bI_{\bchi}
\right].\label{det1}
\end{align}
To continue with the computation, we need to compute the determinant of 
\begin{align*}
\bM&\coloneqq\bI_{\bchi}^T
\begin{pmatrix}
\xi^2(\textbf{1}_n-\bP_\phi^T\bN^{-1}\bP_{\bphi})&\textbf{0} \\
\textbf{0} & \textbf{1}_{{n\choose k}}\\
\end{pmatrix}
\bI_{\bchi}\\
&=\bP_{\bchi} \begin{pmatrix}
\xi^2\textbf{1}_n&\textbf{0} \\
\textbf{0} & \textbf{1}_{{n\choose k}}\\
\end{pmatrix}\bP_{\bchi}^T
-\xi^2\bP_{\bchi}
\begin{pmatrix}
\textbf{1}_n \\
\textbf{0} \\
\end{pmatrix}
\bP_{\phi}^T\bN^{-1}\bP_{\phi}(\textbf{1}_n\,\, \bzero)\bP_{\bchi}^T.
\end{align*}

\noindent\textit{Step 4}: Computing $\det \bM$. For this calculation we introduce an orthonormal basis for $\R^n$, described as $\{\ee_1,...,\ee_n\}$, such that $\ee_1\coloneqq\bchi_1$ and $\ee_2\wedge...\wedge \ee_{k+1}=\bchi_2.$  We can decompose $\rng\nabla \bchi$ into three subspaces which, using the notation in \eqref{replacewedgeID}, are as follows:
\begin{enumerate}
\item $\text{Lsp}\{(\ee_{i},\textbf{0}) : k+2\leq i\leq n\}$, with $n-k-1$ elements,

\item $\text{Lsp}\{(\bzero,\be_i\wedge(\be_j\lrcorner \bchi_2)): 2\leq j\leq k+1, k+2\leq i\leq n\}$, with $k(n-k-1)$ elements,

\item $\text{Lsp}\{\frac{1}{\sqrt{2}}(\be_i,-\be_1\wedge(\be_i\lrcorner\bchi_2)): 2\leq i\leq k+1\}$,  with $k$ elements.

\end{enumerate}
Note that the sum of the dimensions of these subspaces is $m=\text{dim}\, \text{Rng}\, \nabla\bchi$, as expected. Let $\bP_1$, $\bP_2$, and $\bP_3$ be the projections of $\R^{n+{n\choose k}}$ onto each of these three subspaces of $\rng\nabla\bchi$, so that 
\beqn\label{nchidecomp}
\bP_{\bchi}=\bP_1+\bP_2+\bP_3.
\eeqn
The first term in $\bM$ consists of
\[
\bE\coloneqq\bP_{\bchi} 
\begin{pmatrix}
\xi^2\textbf{1}_n&\textbf{0} \\
\textbf{0} & \textbf{1}_{{n\choose k}}\\
  \end{pmatrix}\bP_{\bchi}^T.
  \]
The matrix $[\bE]$ of $\bE$ in terms of the basis $\{\be_1,\dots,\be_n\}$ is
\beqn\label{Ematrix}
 [\bE]=\begin{pmatrix}
\xi^2\textbf{1}_{n-k-1}&\textbf{0}&\textbf{0} \\
\textbf{0}&\textbf{1}_{k(n-k-1)}&\textbf{0}\\
\textbf{0} & \textbf{0}&\frac{\xi^2+1}2\textbf{1}_ k
\end{pmatrix},
\eeqn
which implies $\det\,\bE=\xi^{2(n-k-1)}(\xi^2+1)^k2^{-k}$. To compute the determinant of $\bM$ we now use the matrix determinant lemma in \cite{Harv97}, which states that
\beqn
\det (\bX+\bU\bY\bV^T)=(\det \bX)(\det \bY)\det (\bY^{-1}+\bV^T\bX^{-1}\bU),
\eeqn
with the choices $\bX\coloneqq\bE$, $\bY\coloneqq\bN^{-1}$, and $\bU=\bV\coloneqq\xi\bP_\phi 
\begin{pmatrix}
\bone_n & \textbf{0}
\end{pmatrix}
\bP^T_{\bchi}$
to obtain
\begin{align}
\det \bM=&(\det \bE)(\det \bN^{-1})\det\left(\bN-\xi^2\bP_\phi(\bone_n\,\, \textbf{0})\bP_{\bchi}^T \bE^{-1}\bP_{\bchi} \begin{pmatrix}
\bone_n \\
\textbf{0} \\
\end{pmatrix}
\bP_\phi^T\right)\\
=&\frac{\xi^{2(n-k-1)}(\xi^2+1)^k}{2^kJ}\det\left(\bN-\xi^2\bP_\phi(\bone_n\,\, \textbf{0})\bP_{\bchi}^T \bE^{-1}\bP_{\bchi} \begin{pmatrix}
\bone_n \\
\textbf{0} \\
\end{pmatrix}
\bP_\phi^T\right).
\end{align}

It follows from \eqref{nchidecomp} and the structure of the space associated with $\bP_2$ that
\beqn
\bP_{\bchi}\begin{pmatrix}
\bone_n \\
\textbf{0} \\
\end{pmatrix}
= (\bP_1+\bP_3)\begin{pmatrix}
\bone_n \\
\textbf{0}
\end{pmatrix}.
\eeqn
Using this and \eqref{Ematrix}, we can continue to calculate the determinant of $\bM$ as
\begin{align}
\text{det}\, \bM&=\frac{\xi^{2(n-k-1)}(\xi^2+1)^k}{2^kJ}\det\left(\bN-\bP_\phi(\bone_n\,\, \textbf{0})\bP_1 \begin{pmatrix}
\bone_n \\
\textbf{0} \\
\end{pmatrix}
\bP_\phi^T
-\frac{2\xi^2}{\xi^2+1}
\bP_\phi(\textbf{1}_n\,\, \bzero)\bP_3 \begin{pmatrix}
\bone_n \\
\textbf{0} \\
\end{pmatrix}
\bP_\phi^T
\right)\\
&=\frac{\xi^{2(n-k-1)}}{2^kJ}\det\left((\xi^2+1)\bN-(\xi^2+1)\bP_\phi(\bone_n\,\, \bzero)\bP_1\begin{pmatrix}
\bone_n \\
\textbf{0} \\
\end{pmatrix}
\bP_\phi^T
-2\xi^2
\bP_\phi(\bone_n\,\, \bzero)\bP_3 \begin{pmatrix}
\bone_n \\
\textbf{0} \\
\end{pmatrix}
\bP_\phi^T
\right)\\
&=\frac{\xi^{2(n-k-1)}}{2^kJ}\det\left(\xi^2\Big(\bN-\bP_\phi(\bone_n\,\, \textbf{0})\bP_1 \begin{pmatrix}
\bone_n \\
\bzero \\
\end{pmatrix}
\bP_\phi^T
-2
\bP_\phi(\bone_n\,\, \bzero)\bP_3 \begin{pmatrix}
\bone_n \\
\bzero \\
  \end{pmatrix}
\bP_\phi^T\Big) \right.\\
&\qquad\left.+\Big(\bN-\bP_\phi(\bone_n\,\, \bzero)\bP_1 \begin{pmatrix}
\bone_n \\
\bzero \\
\end{pmatrix}
\bP_\phi^T\Big)
\right).\label{det2}
\end{align}
Consider the two terms appearing in the determinant above:
\[
\bP^\perp_{\bchi_1\wedge\bchi_2}\coloneqq(\bone_n\,\, \bzero)\bP_1 \begin{pmatrix}
\bone_n \\
\bzero \\
\end{pmatrix}\in\Lin(\R^n,\R^n)\quad \mbox{and}\quad \bP_{\bchi_2}\coloneqq2(\bone_n\,\, \bzero)\bP_3 \begin{pmatrix}\bone_n \\
\bzero \\
\end{pmatrix}\in\Lin(\R^n,\R^n),
\]
which are projections onto $[\bchi_1\wedge\bchi_2]^{\perp}$ and $[\bchi_2]$, respectively. Since $\bN=\bP_\phi(
\bone_n-\bchi_1\otimes\bchi_1)\bP_\phi^T$ by \eqref{defN}, it follows that 
\beqn\label{det3}
\bN-\bP_\phi \bP^\perp_{\bchi_1\wedge\bchi_2} \bP_\phi^T-\bP_\phi \bP_{\chi_2}\bP_\phi^T=\bP_\phi(\bone_n-\bchi_1\otimes\bchi_1- \bP^\perp_{\bchi_1\wedge\bchi_2} - \bP_{\bchi_2})\bP_\phi^T=\bzero
\eeqn
and 
\beqn\label{det4}
\bN-\bP_\phi(\bone_n\,\, \bzero)\bP_1 \begin{pmatrix}
\bone_n \\
\bzero \\
\end{pmatrix}
\bP_\phi^T=\bP_\phi(\bone_n-\bchi_1\otimes\bchi_1-\bP^\perp_{\bchi_1\wedge\bchi_2})
\bP_\phi^T=\bP_\phi \bP_{\bchi_2}\bP_\phi^T.
\eeqn
Plugging \eqref{det3} and \eqref{det4} into \eqref{det2} yields
\beqn\label{detMfinal}
\det\bM=\frac{\xi^{2(n-k-1)}}{2^kJ}\det(\bP_\phi \bP_{\bchi_2}\bP_\phi^T).
\eeqn

To compute this final determinant, let $\tt_1,...,\tt_k$ is an orthonormal basis for $T\M$, so that $\vol_{\M}=\tt_1\wedge ....\wedge \tt_k$, then the matrix $[\bP_\phi \bP_{\bchi_2}\bP_\phi^T]$ of $\bP_\phi \bP_{\bchi_2}\bP_\phi^T$ relative to this basis has components
\begin{align*}
[\bP_\phi \bP_{\bchi_2}\bP_\phi^T]_{ij}&=\tt_i\cdot \bP_\phi \bP_{\bchi_2}\bP_\phi^T\tt_j\\
&=\tt_i\cdot\Big(\sum_{\ell=2}^{k+1}\ee_\ell\otimes \ee_\ell\Big)\tt_j\\
&=\sum_{\ell=2}^{k+1}(\tt_i\cdot\ee_\ell)(\ee_\ell\cdot \tt_j)\\
&=\sum_{\ell=1}^{k}B_{i\ell}B_{\ell j},
\end{align*}
where $B_{ij}\coloneqq\ee_{i+1}\cdot \tt_{j}$.  Let $\bB$ be the linear mapping whose matrix has components $B_{ij}$.  It follows from \eqref{kvip} that $\det(\bB)= \vol_{\M}\cdot \bchi_2$, so 
$$\det(\bP_\phi \bP_{\bchi_2}\bP_\phi^T)=\det(\bB\bB)=( \vol_{\M}\cdot \bchi_2)^2.$$  
Thus,
\beqn\label{detM}
\det\bM=\frac{\xi^{2(n-k-1)}}{2^kJ}( \vol_{\M}\cdot \bchi_2)^2.
\eeqn

\noindent\textit{Step 5}: Putting things together. Putting together \eqref{JFcomp}, \eqref{detAcomp}, \eqref{det1} and \eqref{detM} yields
\[
|\det(\nabla F^T\nabla F)|
=\det(\nabla \phi^T\nabla \phi)\det(\nabla \bchi^T\nabla \bchi)\frac{\xi^{2(n-k-1)}}{2^k}( \vol_{\M}\cdot \bchi_2)^2.
\]
Combining this with \eqref{perpchange} and \eqref{surfchange}, the area formula yields the desired result. 
\end{proof}

\begin{lemma}\label{subspint}
Let $\cX,\cY\subseteq\R^n$ be subspaces and suppose $g:\cX\rightarrow \R$ is an integrable function. If $\dim \cX=d$ and $\dim(\cX\cap \cY)=\ell$, then
\beqn\label{subspinteq}
\int_{\cU(\cX)}g(\bx)d\bx=\int_{\cU(\cX\cap \cY)}\int_{\cU(\cX\cap \cY^\perp)}\int_0^{\pi/2}g(\cos\theta \by+\sin\theta \by')\cos^{\ell-1}\theta\sin^{d-\ell-1}\theta d\theta d\by' d\by.
\eeqn
\end{lemma}
\begin{proof} 
Let $\bphi:\mathcal{O}_1\subseteq \R^\ell\rightarrow\cU(\cX\cap \cY)\subseteq\R^n$ and $\bpsi:\mathcal{O}_2\subseteq \R^{d-\ell}\rightarrow \cU(\cX\cap \cY^\perp)\subseteq\R^n$ be charts.  It follows that the function $\bF:\cO_1\times\cO_2\times(0,\pi/2)\rightarrow \cU(\cX)$ defined by
\beqn
\bF(u,v,\theta)=\bphi(u)\cos\theta+\bpsi(v)\sin\theta,\qquad (u,v,\theta)\in\cO_1\times\cO_2\times(0,\pi/2)
\eeqn
is a chart for $\cU(\cX)$.  The charts for $\cU(\cX)$ of this form cover almost all of $\cU(\cX)$.  Thus, it suffices to establish \eqref{subspinteq} when the left-hand side is the range of such a chart. A calculation yields
\[
\nabla \bF=\begin{blockarray}{cccc}
\ell-1 & d-\ell-1 & 1 &   \\
\begin{block}{(ccc)c}
 \nabla\bphi\cos\theta &\nabla \bpsi\sin\theta& -\bphi\sin\theta+\bpsi\cos\theta&n\\
\end{block}
\end{blockarray}\]
and, considering the nullspaces of $\nabla \bphi^T$ and $\nabla \bphi^T,$ we find
\[
\nabla \bF^{T}\nabla \bF=
\begin{blockarray}{cccc}
&\ell-1 & \d-\ell-1 & 1  \\
\begin{block}{(cccc)}
&\cos^2\theta\nabla \bphi^T\nabla \bphi & \bzero  & \bzero \\
  &\bzero &\sin^2\theta \nabla \bpsi^T\nabla \bpsi & \bzero  &\\
  &\bzero & \bzero & 1  \\
\end{block}
\end{blockarray}\,.  
\]
Thus, $\det(\nabla \bF^T\nabla \bF)=(\cos^2\theta)^{\ell-1}(\sin^2\theta)^{d-\ell-1}\det(\nabla \bphi^T\nabla \bphi)\det(\nabla \bpsi^T\nabla \bpsi)$.  The desired result now follows from the area formula. 
\end{proof}

\bibliography{nonmeasure}
\bibliographystyle{is-alpha}

\end{document}